\documentclass[final, 11pt]{article}

\usepackage{enumerate}

\usepackage{amsmath}
\usepackage{amstext,amssymb,amsfonts}
\usepackage{fullpage}
\usepackage{bm}

\usepackage{todonotes}

\newcounter{mynotes}
\usepackage{nameref}
\usepackage[pagebackref,colorlinks,linkcolor=blue,filecolor = blue,
citecolor = blue, urlcolor = blue,pdfstartview=FitH]{hyperref}
\usepackage[capitalize]{cleveref}

\usepackage{amsthm}
\usepackage{thmtools}

\declaretheorem[within=section]{theorem}
\declaretheorem[sibling=theorem]{corollary}

\declaretheorem[sibling=theorem]{lemma}
\declaretheorem[sibling=theorem]{claim}

\declaretheorem[sibling=theorem]{definition}
\declaretheorem[sibling=theorem]{Lemma+Definition}
\declaretheorem[sibling=theorem]{remark}

\declaretheorem[sibling=theorem]{proposition}

\newenvironment{proofof}[1]{\begin{trivlist} \item {\bf Proof
#1:~~}}
  {\qed\end{trivlist}}

\crefname{proposition}{Proposition}{Propositions}
\crefname{conjecture}{Conjecture}{Conjectures}
\crefname{claim}{Claim}{Claims}
\crefname{remark}{Remark}{Remarks}
\crefname{Lemma+Definition}{Lemma+Definition}{Lemma+Definition}

\newcounter{termcounter}
\renewcommand{\thetermcounter}{\Alph{termcounter}}
\crefname{term}{term}{terms}
\creflabelformat{term}{\textup{(#2#1#3)}}

\makeatletter
\def\term{\@ifnextchar[\term@optarg\term@noarg}
\def\term@optarg[#1]#2{%
  \textup{(#1)}%
  \def\@currentlabel{#1}%
  \def\cref@currentlabel{[][2147483647][]#1}%
  \cref@label[term]{#2}}
\def\term@noarg#1{%
  \refstepcounter{termcounter}%
  \textup{(\thetermcounter)}%
  \cref@label[term]{#1}}
\makeatother

\newcommand{\mrm}[1]{\mathrm {#1}}

\newcommand{\msf}[1]{\mathsf {#1}}

\newcommand{\ignore}[1]{}

\newcommand{\defeq}{\stackrel{\mathrm{def}}=}

\newcommand{\poly}{\mathrm{poly}}

\newcommand{\set}[1]{\left\{#1\right\}}

\newcommand{\norm}[1]{\lVert#1\rVert}

\definecolor{DSred}{rgb}{1,0,0}

\renewcommand{\leq}{\leqslant}
\renewcommand{\geq}{\geqslant}
\renewcommand{\ge}{\geqslant}
\renewcommand{\le}{\leqslant}
\renewcommand{\epsilon}{\varepsilon}
\newcommand{\eps}{\epsilon}

\newcommand{\R}{\mathbb{R}}
\newcommand{\C}{\mathbb{C}}
\newcommand{\Z}{\mathbb{Z}}
\newcommand{\N}{\mathbb{N}}
\newcommand{\F}{\mathbb{F}}
\newcommand{\T}{\mathbb{T}}
\newcommand{\U}{\mathbb{U}}

\newcommand{\cB}{\mathcal B}

\newcommand{\cI}{\mathcal I}

\newcommand{\cL}{\mathcal L}
\newcommand{\cM}{\mathcal M}

\newcommand{\cP}{\mathcal P}

\newcommand{\Esymb}{{\bf E}}

\newcommand{\Psymb}{{\bf Pr}}

\DeclareMathOperator*{\E}{\Esymb}

\DeclareMathOperator*{\ProbOp}{\Psymb}

\renewcommand{\Pr}{\ProbOp}

\newcommand{\expo}[1]{{\mathsf{e}\left(#1\right)}}

\def\T{\mathbb{T}}
\def\L{\mathcal{L}}

\def\rank{{\rm{rank}}}
\def\deg{{\rm{deg}}}

\def\cP{\mathcal{P}}

\def\deg{{\mathrm{deg}}}

\def\U{\mathbb{U}}

\def\cB{\mathcal{B}}
\def\LL{\mathcal{L}}
\def\D{\mathbb{D}}
\def\ind{\mathbf{1}}
\def\cL{\mathcal{L}}
\newcommand{\restate}[2]{\medskip
\noindent{\bf #1 (restated).}{\sl #2}}
\def\lc{\mathrm{lc}}
\def\cM{\mathcal{M}}
\def\cI{\mathcal{I}}

\def\poly{\mathrm{Poly}}

\definecolor{Blue}{rgb}{0,0,1}

\title{General systems of linear forms: equidistribution and true complexity}%

\author{Hamed Hatami \thanks{McGill University. \texttt{hatami@cs.mcgill.ca}. Supported by an NSERC, and an FQRNT grant.} \and Pooya Hatami \thanks{University of Chicago. \texttt{pooya@cs.uchicago.edu}.} \and Shachar Lovett \thanks{UC San Diego. \texttt{slovett@ucsd.edu}. Supported by NSF CAREER award 1350481.}}

\date{\today}

\begin{document}

\sloppy

\maketitle

\begin{abstract}
The densities of small linear structures (such as arithmetic progressions) in subsets of Abelian groups can be expressed as certain analytic averages involving linear forms. Higher-order Fourier analysis examines such averages by approximating the indicator function of a subset by a function of bounded number of polynomials. Then, to approximate the average, it suffices to know the joint distribution of the polynomials applied to the linear forms. We prove a near-equidistribution theorem that describes these distributions for the group $\F_p^n$ when $p$ is a fixed prime. This fundamental fact is equivalent to a strong near-orthogonality statement regarding the higher-order characters, and was previously known only under various extra assumptions about the linear forms.

As an application of our near-equidistribution  theorem, we settle a conjecture of Gowers and Wolf on the true complexity of systems of linear forms for the group $\F_p^n$.
\end{abstract}


\section{Introduction}

Gowers' seminal work in combinatorial number theory~\cite{Gow01} initiated an extension of the classical Fourier analysis, called \emph{higher-order Fourier analysis} of Abelian groups. Higher-order Fourier analysis has been very successful in dealing with problems  regarding the densities of small linear structures (e.g. arithmetic progressions) in subsets of Abelian groups. It is possible to express such densities as certain analytic averages. For example, the density of the three term arithmetic progressions in a subset $A$ of an Abelian group $G$ can be expressed as $\E_{x,y\in G}\left[ \ind_A(x)\ind_A(x+y)\ind_A(x+2y)\right].$ More generally, one is often interested in analyzing
\begin{equation}\label{eq:linearAvgGeneralSets}
\E_{x_1,\ldots,x_k\in G} \big[\ind_A(L_1(x_1,\ldots,x_k))\cdots \ind_A(L_m(x_1,\ldots,x_k)) \big],
\end{equation}
where each $L_i$ is a linear form on $k$ variables. Averages of this type are of interest in computer science, additive combinatorics, and analytic number theory.

In this paper we are only interested in the group  $\F^n$ where $\F=\F_p$ for a fixed prime $p$ and $n$ is large. In the classical Fourier analysis of $\F^n$, a function is expressed as a linear combination of the characters of $\F^n$. Note that the  characters of $\F^n$  are exponentials of linear polynomials: for $\alpha \in \F^n$, the corresponding character is defined
as $\chi_\alpha(x) = \expo{\sum_{i=1}^n \alpha_i x_i}$, where $\expo{a}:=e^{\frac{2 \pi i}{p}a}$ for $a \in \F$. In higher-order Fourier analysis, the linear polynomials are replaced by higher degree polynomials, and one would like to approximate a function $f:\F^n \to \C$ by a linear combination of the functions $\expo{P}$, where each $P$ is a polynomial of a certain degree. The existence of such approximations is a consequence of the the so-called ``inverse theorems'' for Gowers norms which are established in a sequence of papers by Bergelson, Green, Samorodnitsky, Szegedy, Tao, and Ziegler~\cite{MR2663409,MR2948765,Balazs,MR2950773,MR2815606,MR2594614,MR2402476}.

Higher-order Fourier expansions are extremely useful in studying averages that are defined through linear structures. To analyze the average in \Cref{eq:linearAvgGeneralSets}, one approximates $1_A \approx \Gamma(P_1,\ldots,P_C)$ where $\Gamma:\F^C \to \C$ is a function  that is applied to a constant number of low degree polynomials $P_1,\ldots,P_C:\F^n \to \F$. Then applying the classical Fourier transform to $\Gamma$ yields the higher-order Fourier expansion
$$\ind_A \approx \sum_{\alpha \in \F^C} \widehat{\Gamma}(\alpha) \; \expo{\sum_{i=1}^n \alpha_i P_i},$$
where the coefficients $\widehat{\Gamma}(\alpha)$ are complex numbers.

\paragraph{Near-Orthogonality and Equidistribution.} One of the important and useful properties of the classical  Fourier characters is that they form an orthonormal basis. For higher-order Fourier expansions to be useful, one needs a similar orthogonality for the higher-order characters $\expo{\sum_{i=1}^n \alpha_i P_i}$ appearing in the expansion, or at least an approximation of it. This approximate orthogonality  is established by Green and Tao~\cite{MR2592422} and Kaufman and Lovett~\cite{kaufman-lovett}, and is in fact equivalent to a  near-equidistribution statement: the polynomials in the approximation $1_A \approx \Gamma(P_1,\ldots,P_C)$ can be chosen in such a way that  the distribution of $(P_1(x),\ldots,P_C(x))$ is close to the uniform  distribution on $\F^C$ when $x$ is chosen uniformly at random from $\F^n$.

However, this is not completely satisfactory, as to study the averages of the form \Cref{eq:linearAvgGeneralSets}, one needs to understand the distribution of the more sophisticated  random variable
$$
\left(
\begin{array}{cccc}
P_1(L_1(X)) & P_2(L_1(X)) & \ldots & P_C(L_1(X)) \\
P_1(L_2(X)) & P_2(L_2(X)) & \ldots & P_C(L_2(X)) \\
\vdots  & & & \vdots \\
P_1(L_m(X)) & P_2(L_m(X)) & \ldots & P_C(L_m(X)) \\
\end{array}
\right),
$$
where $X=(x_1,\ldots,x_k)$ is the uniform random variable taking values in  $(\F^n)^k$.
Since polynomials of a given degree satisfy various linear identities (e.g. every degree one polynomial $P$ satisfies $P(x+y+z)=P(x+y)+P(x+z)-P(x)$), it is no longer possible to choose the polynomials in a way that this  random matrix is almost uniformly distributed on $\F^{C \times m}$. Therefore, in this case one would like to obtain an almost uniform distribution on the points of $\F^{C \times m}$ that are consistent with these linear identities. Note that while~\cite{MR2592422,kaufman-lovett} only say that the entries in each row of this matrix are nearly independent, such a stronger near-equidistribution would in particular imply that  the columns of this matrix are nearly independent.

Hatami and Lovett~\cite{HL11b} established this strong near-equidistribution in the case where the characteristic of the field $\F$ is greater than the degree of the involved polynomials. Bhattacharyya, {\it et al.}~\cite{BFHHL13} extended the result of \cite{HL11b} to the general characteristic case, but under the extra assumption that the system of linear forms is affine, i.e. there is a variable that appears with coefficient $1$ in all the linear forms. Finally, in the present paper,  in \Cref{thm:nearorthogonality} we prove the near-equidistribution statement without any extra assumptions on the linear forms.

\paragraph{A conjecture of Gowers and Wolf.} In dealing with the averages of the form \Cref{eq:linearAvgGeneralSets}  a question arises naturally: Given such an average, what is the smallest $k$ such that there is an approximation of $\ind_A$ with a linear combination of a few higher-order characters of degree at most $k$ that affects the average only negligibly? This question was asked and studied by Gowers and Wolf~\cite{MR2578471} who conjectured a simple characterization for this value, and verified it for the case of large $|\F|$ in~\cite{MR2773103}.  As an application of our near-orthogonality result, we settle the Gowers-Wolf conjecture in full generality on $\F^n$. In the setting of functions on $\mathbb{Z}_N$,  Green and Tao~\cite{MR2815606} established similar results and characterizations.

\paragraph{Homogeneous non-classical polynomials.} The main difficulty in dealing with fields of low characteristic is that in the higher-order Fourier expansions, instead of the exponentials of classical polynomials, one has to work with exponentials of a generalization of them which are referred to as  ``non-classical'' polynomials. Recall that a classical polynomial is  homogeneous if all of its monomials are of the same degree. A useful property of a homogeneous classical polynomial $P(x)$ of degree $d$ is that $P(cx)=c^d P(x)$, for every $c\in \F$. We use this property to extend the definition of homogeneity to non-classical polynomials. An ingredient of the proof of our near-equidistribution result is a statement about non-classical polynomials which we believe is of independent interest. In \Cref{thm:homogeneous-basis} we show that homogeneous multivariate (non-classical) polynomials span the space of multivariate (non-classical) polynomials. We later use this to prove our near-equidistribution results for homogeneous polynomials.

\section{Notation and Preliminaries}

Fix a prime field $\F = \F_p$ for a prime $p \geq 2$.
Throughout the paper, we fix $\zeta\in \F^*$ a generator of $\F^*$. 
Define $|\cdot|$
to be the standard map from $\F$ to $\{0,1,\dots,p-1\} \subset
\Z$. Let $\mathbb{D}$ denote the complex unit disk $\{z \in \C: |z| \le 1\}$.

For integers $a,b$, we let $[a]$ denote the set $\{1,2,\dots,a\}$
and $[a,b]$ denote the set $\{a, a+1, \dots, b\}$.  For real numbers, $\alpha,\sigma, \eps$, we use the shorthand $\sigma= \alpha \pm \eps$ to denote $\alpha-\eps \leq \sigma \leq \alpha + \eps$. The power set of a set $S$ is denoted by $\cP(S)$. The zero element in $\F^n$ is denoted by $\underline{0}$. We will denote by lower case letters, e.g. $x,y$, elements of $\F^n$.  We use capital letters, e.g. $X=(x_1,\ldots, x_k)\in (\F^n)^k$, to denote tuples of variables.

\begin{definition}
A linear form on $k$ variables is a vector $L=(\ell_1,\ldots,\ell_k)\in \F^k$ and it maps $X=(x_1,\ldots,x_k)\in (\F^n)^k$ to $L(X)= \sum_{i=1}^k \ell_i x_i\in \F^n$.
\end{definition}

For a linear form $L=(\ell_1,\ldots,\ell_k)\in \F^k$ we define $|L|\defeq \sum_{i=1}^k |\ell_i|$.

\subsection{Higher-order Fourier Analysis}
We need to recall some definitions and results about  higher-order Fourier analysis. Most of the material in this section is directly quoted from the full version of~\cite{BFHHL13}.

\begin{definition}[Multiplicative Derivative]\label{multderiv}
Given a function $f: \F^n \to \C$ and an element $h \in \F^n$, define
the {\em multiplicative derivative in direction $h$} of $f$ to be the
function $\Delta_hf: \F^n \to \C$ satisfying $\Delta_hf(x) =
f(x+h)\overline{f(x)}$ for all $x \in \F^n$.
\end{definition}

The \emph{Gowers norm} of order $d$ for a function $f: \F^n \to \C$ is the
expected multiplicative derivative of $f$ in $d$ random directions at
a random point.

\begin{definition}[Gowers norm]\label{gowers}
Given a function $f: \F^n \to \C$ and an integer $d \geq 1$, the {\em
  Gowers norm of order $d$}  for $f$ is given by
$$\|f\|_{U^d} = \left|\E_{y_1,\dots,y_d,x \in \F^n} \left[(\Delta_{y_1}
\Delta_{y_2} \cdots \Delta_{y_d}f)(x)\right]\right|^{1/2^d}.$$
\end{definition}
Note that as $\|f\|_{U^1}= |\E\left[f\right]|$ the Gowers norm of order $1$ is only a semi-norm. However for $d>1$, it is not difficult to show that
$\|\cdot\|_{U^d}$ is indeed a norm.

If $f = e^{2\pi i P/p}$ where $P: \F^n \to \F$ is a polynomial
of degree $< d$, then $\|f\|_{U^d} = 1$. If $d < p$ and $\|f\|_\infty \le 1$, then in fact, the
converse holds, meaning that any  function $f: \F^n \to \C$ satisfying
$\|f\|_\infty \le 1$ and $\|f\|_{U^d} = 1$ is of this form. But when $d \geq p$, the converse
is no longer true. In order to characterize functions $f : \F^n \to
 \C$ with $\|f\|_\infty \le 1$ and $\|f\|_{U^d}=1$, one needs to define the notion of {\em
  non-classical polynomials}.

Non-classical polynomials might not be necessarily $\F$-valued. We need to
introduce some notation.
Let $\T$ denote the circle group $\R/\Z$. This is an Abelian group
with group operation denoted $+$. For an integer $k \geq 0$, consider the subgroup $\frac{1}{p^k} \Z/\Z \subseteq \T$.
 Let $\msf{e}: \T \to \C$ denote the character
$\expo{x} = e^{2\pi i x}$.
\begin{definition}[Additive Derivative]\label{addderiv}
Given a function\footnote{We try to adhere to the following convention: upper-case letters (e.g. $F$ and
  $P$) to denote functions mapping from $\F^n$ to $\T$ or to $\F$,
  lower-case   letters (e.g. $f$ and $g$) to denote functions mapping
  from $\F^n$ to $\C$, and upper-case Greek letters (e.g. $\Gamma$ and
$\Sigma$) to denote functions mapping  $\T^C$ to $\T$.} $P:
\F^n \to \T$ and an element $h \in \F^n$, define
the {\em additive derivative in direction $h$} of $f$ to be the
function $D_hP: \F^n \to \T$ satisfying $D_hP(x) = P(x+h) - P(x)$
for all $x \in \F^n$.
\end{definition}
\begin{definition}[Non-classical polynomials]\label{poly}
For an integer $d \geq 0$, a function $P: \F^n \to \T$ is said to be a
{\em non-classical polynomial of degree $\leq d$} (or simply a
{\em polynomial of degree $\leq d$}) if for all $y_1,
\dots, y_{d+1}, x \in \F^n$, it holds that
\begin{equation}\label{eqn:poly}
(D_{y_1}\cdots D_{y_{d+1}} P)(x) = 0.
\end{equation}
The {\em degree} of $P$ is the smallest $d$ for which the above holds.
A function $P : \F^n \to \T$ is said to be a {\em classical polynomial of degree
$\leq d$} if it is a non-classical polynomial of degree $\leq d$
whose image is contained in $\frac{1}{p} \Z/\Z$.
\end{definition}

It is a direct consequence of the definition that a function $f :
\F^n \to \C$ with $\|f\|_\infty \leq 1$ satisfies $\|f\|_{U^{d+1}} =
1$ if and only if $f = \expo{P}$ for a
(non-classical) polynomial $P: \F^n \to \T$ of degree $\leq d$.  We denote by $\poly(\F^n\rightarrow \T)$  and $\poly_{\leq d}(\F^n\rightarrow \T)$, respectively, the set of all non-classical polynomials, and the ones of degree at most $d$.

The following lemma of Tao and Ziegler~\cite{MR2948765} shows that a classical
polynomial $P$ of degree $d$ must always be of the form $x \mapsto \frac{|Q(x)|}{p}$, where $Q : \F^n \to \F$ is a polynomial (in the usual sense) of degree $d$, and $|\cdot|$ is the standard map from $\F$ to
$\set{0,1,\dots,p-1}$. This lemma also characterizes the structure of non-classical
polynomials.

\begin{lemma}[Lemma 1.7 in \cite{MR2948765}]\label{struct}
A function $P: \F^n \to \T$ is a polynomial of degree $\leq d$ if and
only if $P$ can be represented as
$$P(x_1,\dots,x_n) = \alpha + \sum_{0\leq d_1,\dots,d_n< p; k \geq 0:
  \atop {0 < \sum_i d_i \leq d - k(p-1)}} \frac{ c_{d_1,\dots, d_n,
  k} |x_1|^{d_1}\cdots |x_n|^{d_n}}{p^{k+1}} \mod 1,
$$
for a unique choice of $c_{d_1,\dots,d_n,k} \in \set{0,1,\dots,p-1}$
and $\alpha \in \T$.  The element $\alpha$ is called the {\em
  shift} of $P$, and the largest integer $k$ such that there
exist $d_1,\dots,d_n$ for which $c_{d_1,\dots,d_n,k} \neq 0$ is called
the {\em depth} of $P$. A depth-$k$ polynomial $P$ takes values in a coset of the subgroup $\U_{k+1}\defeq \frac{1}{p^{k+1}} \Z/\Z$. Classical polynomials correspond to
polynomials with $0$ shift and $0$ depth.
\end{lemma}

Note that \Cref{struct} immediately implies the following important
observation\footnote{Recall that $\T$ is an additive
 group. If $n \in \Z$ and $x \in \T$, then $nx$ is shorthand for
 $x + \cdots + x$ if $n \geq 0$ and
 $-x - \cdots - x$ otherwise, where there are $|n|$ terms in both expressions.}:
\begin{remark}
If $Q: \F^n \to \T$ is a polynomial of degree $d$ and depth $k$, then
$pQ$ is a polynomial of degree $\max(d-p+1, 0)$ and depth $k-1$. In
other words, if $Q$ is classical, then $pQ$ vanishes, and otherwise,
its degree decreases by $p-1$ and its depth by $1$. Also, if $\lambda
\in [1, p-1]$ is an integer, then $\deg(\lambda Q) = d$ and
$\mrm{depth}(\lambda Q) = k$.
\end{remark}
For convenience of exposition, we will assume throughout this
paper that the shifts of all polynomials are zero. This can be done
without affecting any of the results in this work. Hence, all
polynomials of depth $k$ take values in $\U_{k+1}$.

Given a degree-$d$ non-classical polynomial $P$, it is often useful to consider the properties of its $d$-th derivative. Motivated by this, we give the following definition.

\begin{definition}[Derivative Polynomial]\label{dfn:derivativepolynomial}
Let $P:\F^n\rightarrow \T$ be a degree-d polynomial, possibly non-classical. Define the derivative polynomial
$\partial P:(\F^n)^d\rightarrow \T$ by the following formula
$$
\partial P(h_1,\ldots,h_d)\defeq D_{h_1}\cdots D_{h_d}P(0),
$$
where $h_1,\ldots,h_d\in \F^n$.\footnote{Notice since $P$ is a degree $d$ polynomial, $D_{h_1}\ldots D_{h_d}P(x)$ does not depend on $x$ and thus we have the identity $\partial P(h_1,\ldots,h_d)=D_{h_1}\ldots D_{h_d}P(x)$ for any choice of $x\in \F^n$.}
Moreover for $k< d$ define
$$
\partial_kP(x, h_1,\ldots,h_k)\defeq D_{h_1}\cdots D_{h_k}P(x).
$$
\end{definition}

The following lemma shows some useful properties of the
derivative polynomial.
\begin{lemma}\label{lem:derivativepolynomial}
\label{lem:derivativepoly}
Let $P:\F^n\rightarrow \T$ be a degree-d (non-classical) polynomial. Then the
polynomial $\partial P(h_1,\ldots,h_d)$ is
\begin{itemize}
\item[(i)] multilinear: $\partial P$ is additive in each $h_i$.
\item[(ii)] invariant under permutations of $h_1,\ldots, h_d$.
\item[(iii)] a classical nonzero polynomial of degree $d$.
\item[(iv)] homogeneous: All its monomials are of degree $d$.
\end{itemize}
\end{lemma}
Notice that by multilinear we mean additive in each direction $h_i$, which is not the
usual use of the term ``multilinear''.

\begin{proof}
The proof follows by the properties of the additive derivative $D_h$. Multilinearity of $\partial P$
follows from linearity of the additive derivative, namely for every function $Q$ and directions $h_1,h_2$ we have the identity $D_{h_1+h_2}Q(x)= D_{h_1}Q(x)+D_{h_2}Q(x+h_1)$. The invariance under permutations of $h_1,\ldots, h_d$ is a result of commutativity of the additive derivatives. Since $P$ is
a degree-$d$ (non-classical) polynomial, $\partial P$ is nonzero by definition.  Notice that since $D_{\underline{0}}Q\equiv 0$ for any function $Q$, we have $\partial P(h_1,\ldots,h_d) = 0$ if any of $h_i$ is equal to zero. Hence every monomial of $\partial P$ must depend on all $h_i$'s. The properties (iii) and (iv) now follow from
this and the fact that $\deg(\partial P)\leq d$ and thus each monomial has exactly one variable from each $h_i$.
\end{proof}

%
%

\subsection{Rank of a Polynomial}
We will often need to study Gowers norms of exponentials of polynomials. As we describe below if this analytic
quantity is non-negligible, then there is an algebraic explanation for it: it is possible to decompose the polynomial as a function of a constant number of low-degree polynomials. To state this rigorously, let us define the notion of {\em rank} of a polynomial.

\begin{definition}[Rank of a polynomial]\label{def:rankpoly}
Given a polynomial $P : \F^n \to \T$ and an integer $d > 1$, the {\em $d$-rank} of
$P$, denoted $\msf{rank}_d(P)$, is defined to be the smallest integer
$r$ such that there exist polynomials $Q_1,\dots,Q_r:\F^n \to \T$ of
degree $\leq d-1$ and a function $\Gamma: \T^r \to \T$ satisfying
$P(x) = \Gamma(Q_1(x),\dots, Q_r(x))$. If $d=1$, then
$1$-rank is defined to be $\infty$ if $P$ is non-constant and $0$
otherwise.

The {\em rank} of a polynomial $P: \F^n \to \T$ is its $\deg(P)$-rank. We say $P$ is $r$-regular if $\msf{rank}(P) \ge r$.
\end{definition}

Note that for integer $\lambda \in [1, p-1]$,  $\mrm{rank}(P) =
\mrm{rank}(\lambda P)$. We also define the following weaker analytical notion of uniformity for a polynomial.

\begin{definition}[Uniformity]\label{def:uniformpoly}
Let $\epsilon>0$ be a real. A degree-$d$ polynomial $P:\F^n\rightarrow \T$ is said to be $\epsilon$-uniform if
$$
\norm{\expo{P}}_{U^d}<\epsilon.
$$
\end{definition}

The following theorem of Tao and Ziegler shows that high rank polynomials have small Gowers norm.

\begin{theorem}[Theorem 1.20 of \cite{MR2948765}]\label{thm:taoziegler}
For any $\epsilon>0$ and integer $d>0$, there exists an integer $r(d,\epsilon)$ such that the following is true. For any polynomial $P:\F^n\rightarrow \T$ of degree $\leq d$, if $\norm{\expo{P}}_{U^d}\geq \epsilon$, then $\rank_d(P)\leq r$.
\end{theorem}

This immediately implies that a regular polynomial is also uniform.

\begin{corollary}\label{cor:rankvsuniformity}
Let $\epsilon, d,$ and $r(d,\epsilon)$ be as in \Cref{thm:taoziegler}. Every $r$-regular polynomial $P$ of degree $d$ is also $\epsilon$-uniform.
\end{corollary}

\subsection{Polynomial Factors}
A high-rank polynomial of degree $d$ is, intuitively, a ``generic''
degree-$d$ polynomial. There are no unexpected ways to decompose it
into lower degree polynomials. Next, we will formalize the
notion of a generic collection of polynomials. Intuitively, it should
mean that there are no unexpected algebraic dependencies among the
polynomials. First, we need to set up some notation.

\begin{definition}[Factors] If $X$ is a finite set then by a \emph{factor} $\cB$ we mean simply a
partition of $X$ into finitely many pieces called \emph{atoms}.
\end{definition}

A function $f:X \to \C$ is called \emph{$\cB$-measurable} if it is constant on atoms of $\cB$. For any function $f : X \to \C$, we may define
the conditional expectation
$$\E[f|\cB](x)=\E_{y \in \cB(x)}[f(y)],$$
where $\cB(x)$ is the unique atom in $\cB$ that contains $x$. Note that $\E[f|\cB]$ is $\cB$-measurable.

A finite collection of functions $\phi_1,\ldots,\phi_C$ from $X$ to some other space $Y$ naturally define a factor $\cB=\cB_{\phi_1,\ldots,\phi_C}$ whose atoms are sets of the form $\{x: (\phi_1(x),\ldots,\phi_C(x))= (y_1,\ldots,y_C) \}$ for some $(y_1,\ldots,y_C) \in Y^C$. By an abuse of notation
we also use $\cB$ to denote the map $x \mapsto (\phi_1(x),\ldots,\phi_C(x))$, thus also identifying the atom containing $x$ with
$(\phi_1(x),\ldots,\phi_C(x))$.

\begin{definition}[Polynomial factors]\label{factor}
If $P_1, \dots, P_C:\F^n \to \T$ is a sequence of polynomials, then the factor $\cB_{P_1,\ldots,P_C}$ is called a {\em polynomial factor}.
\end{definition}

The {\em complexity} of $\cB$, denoted $|\cB|:=C$, is the number of defining polynomials. The {\em degree} of $\cB$ is the maximum degree among its defining polynomials $P_1,\ldots,P_C$. If $P_1,\ldots,P_C$ are of depths $k_1,\ldots,k_C$, respectively, then the number of atoms of $\cB$ is at most $\prod_{i=1}^C p^{k_i+1}$.

\begin{definition}[Rank and Regularity]
A polynomial factor $\cB$ defined by a sequence of polynomials $P_1,\ldots,P_C:\F^n \rightarrow \T$ with respective depths $k_1,\ldots,k_C$ is said to have rank $r$ if $r$ is the least integer for which there exists $(\lambda_1, \ldots, \lambda_C)\in \Z^C$, with $(\lambda_1 \mod p^{k_1+1}, \ldots, \lambda_C \mod p^{k_C+1})\neq 0^C$, such that $\rank_d(\sum_{i=1}^C \lambda_iP_i) \leq r$, where $d=\max_i \deg(\lambda_i P_i)$.

Given a polynomial factor $\cB$ and a function $r:\Z_{>0}\rightarrow \Z_{>0}$, we say that $\cB$ is $r$-regular if $\cB$ is of rank larger than $r(|\cB|)$.
\end{definition}

Notice that by the above definition of rank for a degree-$d$ polynomial $P$ of depth $k$ we have
$$
\rank(\{P\})= \min\left\{ \rank_d(P),\rank_{d-(p-1)}(pP),\ldots, \rank_{d-k(p-1)}(p^kP)\right\}.
$$

We also define the following weaker analytical notion of uniformity for a factor along the same lines as \Cref{def:uniformpoly}.

\begin{definition}[Uniform Factor]
Let $\epsilon>0$ be a real. A polynomial factor $\cB$ defined by a sequence of polynomials $P_1,\ldots, P_C:\F^n \rightarrow \T$ with respective depths $k_1,\ldots, k_C$ is said to be $\epsilon$-uniform if for every collection $(\lambda_1,\ldots, \lambda_C)\in \Z^C$, with $(\lambda_1 \mod p^{k_1+1}, \ldots, \lambda_C \mod p^{k^C+1})\neq 0^C$
$$
\left\|\expo{\sum_i \lambda_i P_i}\right\|_{U^d}<\epsilon,
$$
where $d= \max_{i} \deg(\lambda_i P_i)$.
\end{definition}

\begin{remark}\label{remark:uniformityvsrank}
Similar to \Cref{cor:rankvsuniformity} it also follows from \Cref{thm:taoziegler}
that an $r$-regular degree-$d$ factor $\cB$ is also $\epsilon$-uniform when $r=r(d,\epsilon)$ is as in \Cref{thm:taoziegler}.
\end{remark}

\subsubsection{Regularization of Factors}
Due to the generic properties of regular factors, it is often useful to {\em refine} a given polynomial
factor to a regular one~\cite{MR2948765, BFL12, BFHHL13}. We will first formally define what we mean by refining a polynomial factor.

\begin{definition}[Refinement] \label{refine}
A factor $\cB'$ is called a {\em refinement} of $\cB$, and
denoted $\cB' \succeq \cB$, if the
induced partition by $\cB'$ is a combinatorial refinement of the partition
induced by $\cB$. In other words, if for every $x,y\in \F^n$,
$\cB'(x)=\cB'(y)$ implies $\cB(x)=\cB(y)$. We will write $\cB\succeq_{\text{syn}} \cB'$, if the polynomials defining $\cB'$ extend that of $\cB$.\ignore{\footnote{One sometimes needs to distinguish between semantic and syntactic refinements, where the former is as \Cref{refine} and the latter means that the polynomials defining $\cB'$ extends that of $\cB$. Being a syntactic refinement is stronger than being a
semantic refinement. But observe that if $\cB'$ is a semantic
refinement of $\cB$, then there exists a
syntactic refinement $\cB''$ of $\cB$ that induces the same
partition of $\F^n$, and for which $|\cB''|\leq
|\cB'|+|\cB|$, because we can define $\cB''$ by
simply adding the defining polynomials of $\cB$ to those of $\cB'$.}}
\end{definition}

The following lemma from \cite{BFHHL13} which uses a regularization theorem of \cite{MR2948765} allows one to regularize a given factor to any desired regularity.
\begin{lemma}[Polynomial Regularity Lemma~\cite{BFHHL13}]\label{factorreg}
Let $r: \Z_{>0} \to \Z_{>0}$ be a non-decreasing function and $d > 0$
be an integer. Then, there is a function
$C_{\ref{factorreg}}^{(r,d)}: \Z_{>0} \to \Z_{>0}$ such
that the following is true. Suppose $\cB$ is a factor defined by
polynomials $P_1,\dots, P_C : \F^n \to \T$ of degree at most $d$.
Then, there is an $r$-regular factor $\cB'$ consisting of  polynomials
$Q_1, \dots, Q_{C'}: \F^n \to \T$ of degree $\leq d$ such that $\cB'
\succeq \cB$ and $C' \leq  C_{\ref{factorreg}}^{(r,d)}(C)$.
\end{lemma}

\subsection{Decomposition Theorems}\label{sec:decompose}
An important application of the inverse theorems are the ``decomposition theorems''~\cite{MR2669681, Tao07, MR2592422}.
These theorems allow one to express a given function $f$ with certain properties as a sum $\sum_{i=1}^k g_i$, where each $g_i$ has certain desired structural properties.  We refer the interested reader to \cite{MR2669681} and \cite{MR2359469} for a detailed discussion of this subject.  The following decomposition theorem is a consequence of an inverse theorem for Gowers norms (\cite[Theorem 1.11]{MR2948765}).

\begin{theorem}[Strong Decomposition Theorem for Multiple Functions]
\label{thm:strongdecomposition}
Let $m,d\geq 1$ be integers, $\delta>0$ a parameter, and let $r:\N\rightarrow \N$ be an arbitrary growth function. Given any functions $f_1,\ldots,f_m:\F^n\rightarrow \D$, there exists a decomposition
$$
f_i=g_i+h_i,
$$
such that for every $1 \le i \le m$,
\begin{enumerate}
\item $g_i= \E[f_i|\cB]$, where $\cB$ is an $r$-regular polynomial factor of degree at most $d$ and complexity $C\leq C_{\mathrm{max}}(p, m, d, \delta, r(\cdot)) $,
\item $\norm{h_i}_{U^{d+1}}\leq \delta$.
\end{enumerate}
\end{theorem}

\section{Main Results}

\subsection{Homogeneous Polynomials}\label{sec:homogeneouspolynomials}
Recall that a classical polynomial is called homogeneous if all of its monomials are of the same degree. Trivially a homogeneous classical polynomial $P(x)$ satisfies $P(cx)=|c|^d P(x)$ for every $c\in \F$. We will use this property to define the class of  non-classical homogeneous polynomials.

\begin{definition}[Homogeneity]\label{def:homogeneity}
A (non-classical) polynomial $P:\F^n \to \T$ is called homogeneous if for every $c \in \F$ there exists a $\sigma_c\in \Z$ such that $P(cx) = \sigma_c P(x) \mod 1$ for all $x$.
\end{definition}

\begin{remark}\label{remark:homogeneity}
It is not difficult to see that $P(cx)=\sigma_c P(x) \mod 1$ implies that $\sigma_c = |c|^{\deg(P)} \mod p$, a property that we will use later. Indeed for $d=\deg(P)$, we have $0 = \partial_d (P(cx) - \sigma_c P(x))= (|c|^d-\sigma_c) \partial_d P(x) \mod 1$. This, since $\partial_d P(x)$ is a nonzero degree-$d$ classical polynomial,  implies $\sigma_c = |c|^d \mod p$.
\end{remark}

Notice that for a polynomial $P$ to be homogeneous it suffices that there exists $\sigma\in \Z$ for which $P(\zeta x)= \sigma P(x) \mod 1$, where $\zeta$ is a generator of $\F^*$. If $P$ has depth $k$, then we can 
assume that $\sigma \in \Z_{p^{k+1}}$, as $p^{k+1} P \equiv 0$. The following lemma shows that $\sigma$ is uniquely determined for all homogeneous polynomials of degree $d$ and depth $k$. Henceforth, we will denote this unique value by $\sigma(d,k)$.

\begin{lemma}
For every $d$ and $k$, there is a unique $\sigma=\sigma(d,k)\in \Z_{p^{k+1}}$, such that for every homogeneous polynomial $P$ of degree $d$ and depth $k$, $P(\zeta x)= |\sigma| P(x) \mod 1$, where $|\cdot|$ is the natural map from $\Z_{p^{k+1}}$ to $\{0,1,\ldots,p^{k+1}-1\} \subset \Z$.
\end{lemma}
\begin{proof}
Let $P$ be a homogeneous polynomial of degree $d$ and depth $k$, and let $\sigma\in \Z_{p^{k+1}}$ be such that $P(\zeta x)= |\sigma| P(x) \mod 1$. By \cref{remark:homogeneity} we know that $|\sigma| = |\zeta|^d \mod p$. We also observe that $P(x)= P(\zeta^{p-1} x)= |\sigma|^{p-1} P(x) \mod 1$ from which it follows that $\sigma^{p-1}=1$. We claim that $\sigma\in \Z_{p^{k+1}}$ is uniquely determined by the two properties 
\begin{enumerate}
\item[i.] $|\sigma| = |c|^d \mod p$, and
\item[ii.] $\sigma^{p-1}= 1$.
\end{enumerate}
Suppose to the contrary that there are two nonzero values $\sigma_1,\sigma_2 \in \Z_{p^{k+1}}$ that satisfy the above two properties, and choose $t\in \Z_{p^{k+1}}$ such that $\sigma_1= t\sigma_2$. It follows from (i) that $t=1 \mod p$ and from (ii) that $t^{p-1}=1$. We will show that $t=1$ is the only possible such value in $\Z_{p^{k+1}}$.

Let $a_1,\ldots, a_{p^{k}}\in \Z_{p^{k+1}}$ be all the possible solutions to $x=1 \mod p$ in $\Z_{p^{k+1}}$. Note that $ta_1,\ldots, ta_{p^{k}}$ is just a permutation of the first sequence and thus
$$
t^{p^{k}}\prod a_i = \prod a_i.
$$
Consequently $t^{p^{k}} = 1$, which combined with $t^p=t$ implies  $t=1$.
\end{proof}

\Cref{struct} allows us to express every (non-classical) polynomial as a linear span of monomials of the form $\frac{ |x_1|^{d_1}\cdots |x_n|^{d_n}}{p^{k+1}}$. Unfortunately, unlike in the classical case, these monomials are not necessarily homogeneous, and for some applications it is important to express a polynomial as a linear span of homogeneous polynomials.
We show that this is possible as homogeneous multivariate (non-classical) polynomials linearly span the space of multivariate (non-classical) polynomials. We will present the proof of this theorem in \Cref{sec:proofhomogeneous-basis}.

\begin{theorem}\label{thm:homogeneous-basis}
There is a basis  for $\poly(\F^n\rightarrow \T)$ consisting only of homogeneous multivariate polynomials.
\end{theorem}

\Cref{thm:homogeneous-basis} allows us to make the extra assumption in the strong decomposition theorem (\Cref{thm:strongdecomposition}) that the resulting polynomial factor $\cB$ consists only of homogeneous polynomials.

\begin{corollary}\label{cor:homogeneous-strong-decomposition}
Let $d\geq 1$ be an integer, $\delta>0$ a parameter, and let $r:\N\rightarrow \N$ be an arbitrary growth function. Given any functions $f_1,\ldots,f_m:\F^n\rightarrow \D$, there exists a decomposition
$$
f_i=g_i+h_i,
$$
such that or every $1 \le i \le m$,
\begin{enumerate}
\item $g_i= \E[f_i|\cB]$, where $\cB$ is an $r$-regular polynomial factor of degree at most $d$ and complexity $C\leq C_{\mathrm{max}}(p, d, \delta, r(\cdot))$, moreover $\cB$ only consists of homogeneous polynomials.
\item $\norm{h_i}_{U^{d+1}}\leq \delta$.
\end{enumerate}
\end{corollary}

\subsection{Strong Near-Orthogonality}
As mentioned in the introduction the main result of this paper is a new near-orthogonality result for polynomial factors of high rank. Such a statement was proved in \cite{MR2592422, MR2948765} for systems of linear forms corresponding to repeated derivatives or equivalently Gowers norms, in \cite{HL11b} for the case when the field is of high characteristic but with arbitrary system of linear forms and in \cite{BFHHL13} for systems of affine linear forms. In \Cref{thm:nearorthogonality} we establish the near-orthogonality over any arbitrary system of linear forms. Before stating this theorem we need to introduce the notion of consistency.

\begin{definition}[Consistency]\label{def:consistent}
Let $\cL = \{L_1, \dots, L_m\}$ be a system of linear forms. A vector $(\beta_1, \dots, \beta_m) \in \T^m$ is said to be {\em $(d,k)$-consistent with $\cL$} if there exists a homogeneous polynomial $P$ of degree $d$ and depth $k$ and a point $X$ such that $P(L_i(X))=\beta_i$ for every $i \in [m]$. Let $\Phi_{d,k}(\cL)$ denote the set of all such vectors.
\end{definition}

It is immediate from the definition that $\Phi_{d,k}(\cL) \subseteq \U_{k+1}^{m}$ is a subgroup of $\T^m$, or more specifically, a subgroup of $\U_{k+1}^m$. Let
$$\Phi_{d,k}(\cL)^\perp := \left\{(\lambda_1,\ldots,\lambda_m) \in  \Z^m \ : \ \forall (\beta_1,\ldots,\beta_m) \in  \Phi_{d,k}(\cL), \ \sum \lambda_i \beta_i = 0   \right\}.$$
Equivalently $\Phi_{d,k}(\cL)^\perp$ is the set of all $(\lambda_1,\ldots,\lambda_m) \in  \Z^m$ such that $\sum_{i=1}^m \lambda_i P(L_i(X)) \equiv 0$ for every homogeneous polynomial $P$ of degree $d$ and depth $k$.

\begin{theorem}[Near Orthogonality over Linear Forms]\label{thm:nearorthogonality}
Let $L_1,\ldots, L_m$ be linear forms on $\ell$ variables and let $\cB=(P_1,\ldots,P_C)$ be an $\epsilon$-uniform polynomial factor for some $\epsilon\in (0,1]$ defined only by homogeneous polynomials. For every tuple $\Lambda$ of integers $(\lambda_{i,j})_{i\in [C], j\in [m]}$, define $P_\Lambda:(\F^n)^\ell\rightarrow \T$ as
$$
P_{\Lambda}(X)= \sum_{i\in [C], j\in [m]} \lambda_{i,j} P_i(L_j(X)).
$$
Then one of the following two statements holds: 
\begin{itemize}
\item $P_\Lambda \equiv 0.$
\item $P_\Lambda$ is non-constant and $\left| \E_{X\in (\F^n)^\ell} [\expo{P_\Lambda}] \right|< \epsilon$.
\end{itemize}
Furthermore $P_\Lambda \equiv 0$ if and only if for every $i \in [C]$, we have $(\lambda_{i,j})_{j \in [m]} \in \Phi_{d_i,k_i}(\cL)^\perp$ where $d_i,k_i$ are the degree and depth of $P_i$, respectively.  
\end{theorem}
We will present the proof of \Cref{thm:nearorthogonality} in \Cref{sec:proofofnearorthogonality}.

\begin{remark}
By \Cref{cor:rankvsuniformity} the assumption of $\epsilon$-uniformity in \Cref{thm:nearorthogonality} is satisfied for every factor $\cB$ of rank at least $r_{\ref{thm:taoziegler}}(d,\epsilon)$. However, we would like to point out that in \Cref{thm:nearorthogonality} by using  the assumption of $\epsilon$-uniformity  instead of the assumption of high rank, we are able to achieve the quantitative bound of $\epsilon$ on the bias of $P_\Lambda$.
\end{remark}

\begin{remark}
In \Cref{thm:nearorthogonality} in the second case where $P_\Lambda$ is non-constant, it is possible to deduce a more general statement  that  $\norm{ \expo{P_\Lambda}}_{U^t}^{2^t}< \epsilon$ for every $t \le \deg(P_\Lambda)$. Indeed assume that $P_\Lambda$ is non-constant, and consider the derivative
\begin{equation}\label{eq:partialplambda}
D_{Y_1} \ldots D_{Y_t} P_\Lambda(X)= \sum_{i\in [C], j\in [m]} \lambda_{i,j} \sum_{S\subseteq [t]} (-1)^{|S|} P_i(L_j(X+\sum_{r\in S} Y_r)),
\end{equation}
where $Y_i:=(y_{i,1},\ldots,y_{i,\ell}) \in (\F^n)^\ell$. Notice that for every choice of $j\in [m]$ and $S\subseteq [t]$, $L_j(X+ \sum_{r\in S} Y_r)$ is an application of a linear form on the vector $\left(x_1,\ldots,x_\ell, y_{1,1},\ldots,y_{1,\ell},\ldots,y_{t,1},\ldots, y_{t,\ell}\right)\in (\F^n)^{(t+1)\ell}$, and  since  by~\Cref{lem:derivativepolynomial} the polynomial $\partial P_\Lambda$ is nonzero, \Cref{thm:nearorthogonality} implies
$$
\norm{\expo{P_\Lambda}}_{U^t}^{2^t}= \E_{h_1,\ldots, h_t}\left[\expo{\partial_t P_\Lambda(h_1,\ldots,h_t)}\right]\leq \epsilon.
$$
\end{remark}

It is well-known that statements similar to that of \Cref{thm:nearorthogonality} imply ``near-equidistributions'' of the joint distribution of the polynomials applied to linear forms. Consider a highly uniform polynomial factor of degree $d > 0$,  defined by a tuple of homogeneous polynomials $P_1, \dots, P_C: \F^n \to \T$ with respective degrees $d_1, \dots, d_C$ and depths $k_1, \dots, k_C$, and let  $\cL=(L_1, \dots, L_m)$ be a collection of linear forms on $\ell$ variables. As we mentioned earlier, we are interested in the distribution of the random matrix
\begin{equation}
\label{eq:matrixPoly}
\left(
\begin{array}{cccc}
P_1(L_1(X)) & P_2(L_1(X)) & \ldots & P_C(L_1(X)) \\
P_1(L_2(X)) & P_2(L_2(X)) & \ldots & P_C(L_2(X)) \\
\vdots  & & & \vdots \\
P_1(L_m(X)) & P_2(L_m(X)) & \ldots & P_C(L_m(X)) \\
\end{array}
\right),
\end{equation}
where $X$ is the uniform random variable taking values in $(\F^n)^\ell$. Note that by the definition of consistency, for every $1 \le i \le C$, the $i$-th column of this matrix must belong to $\Phi_{d_i,k_i}(\cL)$. \Cref{equidist} below says that \Cref{eq:matrixPoly} is ``almost'' uniformly distributed over the set of all matrices satisfying this condition. The proof of \Cref{equidist} is standard and is identical to the proof of \cite[Theorem 3.10]{BFHHL13} with the only difference that it uses \Cref{thm:nearorthogonality} instead of the weaker near-orthogonality theorem of \cite{BFHHL13}.

\begin{theorem}[Near-equidistribution]\label{equidist}
Given $\eps > 0$, let $\cB$ be an $\epsilon$-uniform polynomial factor of
degree $d > 0$ and  complexity $C$, that is defined by a tuple of homogeneous
polynomials $P_1, \dots, P_C: \F^n
\to \T$ having respective degrees $d_1, \dots, d_C$ and
depths $k_1, \dots, k_C$.
Let $\cL=(L_1, \dots, L_m)$ be a collection of linear forms on $\ell$ variables.

Suppose $(\beta_{i,j})_{i \in [C], j \in [m]} \in \T^{C \times m}$  is such that $(\beta_{i,1},\ldots,\beta_{i,m}) \in \Phi_{d_i,k_i}(\cL)$ for every $i \in [C]$. Then
$$
\Pr_{X \in (\F^n)^{\ell}}\left[P_i(L_j(X)) = \beta_{i,j}~
 \forall i \in [C],j \in [m] \right] =
\frac{1}{K} \pm \eps,
$$
where $K= \prod_{i=1}^C |\Phi_{d_i,k_i}(\cL)|$.
\end{theorem}
\begin{proof}
We have
\begin{align*}
&\Pr[P_i(L_j(X)) = \beta_{i,j}~
\forall i \in [C], \forall j \in [m]]= \E\left[ \prod_{i, j}\frac{1}{p^{k_i+1}} \sum_{
    \lambda_{i, j}=0}^{p^{k_i+1}-1} \expo{\lambda_{i, j}
    \big(P_{i}(L_j(X)) - \beta_{i,j}\big)}\right]\\
&= \left(\prod_{i \in [C]}p^{-(k_i+1)}\right)^m \sum_{(\lambda_{i,j})}
\expo{-\sum_{i,j} \lambda_{i,j}\beta_{i,j}}
\E \left[\expo{\sum_{i \in [C],j \in [m]} \lambda_{i,j}P_{i}(L_j(X))}\right],
\end{align*}
where the outer sum is over $(\lambda_{i,j})_{i \in [C], j \in [m]}$ with $\lambda_{i,j} \in [0,p^{k_i + 1}-1]$. Let $\Lambda_i=\Phi_{d_i,k_i}(\cL)^\perp \cap [0,p^{k+1}-1]^m$, and note that $|\Lambda_i||\Phi_{d_i,k_i}|= p^{m(k_i+1)}$.   Since $(\beta_{i,1},\ldots,\beta_{i,m}) \in \Phi_{d_i,k_i}(\cL)$ for every $i \in [C]$, it follows that  $\sum_{i,j} \lambda_{i,j}\beta_{i,j}=0$ if $(\lambda_{i,1},\ldots,\lambda_{i,m}) \in \Lambda_i$ for all $i \in [C]$.  If the latter holds, then the expected value in the above expression is $0$, and otherwise  by  \cref{thm:nearorthogonality}, it is bounded by $\epsilon$. Hence the above expression can be approximated by
$$ p^{-m \sum_{i=1}^C(k_i+1)} \cdot \left(\prod_{i=1}^C |\Lambda_i|~ \pm~
  \eps p^{m \sum_{i=1}^C(k_i+1)}\right)= \frac{1}{K}\pm \epsilon.$$
\end{proof}

\subsection{On a Theorem of Gowers and Wolf}\label{sec:gowerswolf}

Let $A$ be a subset of $\F^n$ with the indicator function $\ind_A:\F^n \to \{0,1\}$. As mentioned in the introduction, \Cref{eq:linearAvgGeneralSets} equals the probability that $L_1(X),\ldots,L_m(X)$ all fall in $A$, where $X \in (\F^n)^k$ is chosen uniformly at random. Roughly speaking, we say $A \subseteq \F^n$ is \emph{pseudorandom} with regards to $\mathcal{L}$ if
$$
\E_{X} \left[ \prod_{i=1}^m \ind_A(L_i(X)) \right] \approx \left( \frac{|A|}{p^n} \right)^m;
$$
That is if the probability that all $L_1(X),\ldots,L_m(X)$ fall in $A$ is close to what we would expect if $A$ was a random subset of $\F^n$ of cardinality $|A|$. Let $\alpha := |A|/p^n$ be the density of $A$, and define $f:=\ind_A - \alpha$. We have
$$
\E_{X} \left[ \prod_{i=1}^m \ind_A(L_i(X)) \right]=\E_{X} \left[ \prod_{i=1}^m \left(\alpha+ f(L_i(X))\right) \right]=\alpha^m + \sum_{S \subseteq [m], S \ne \emptyset} \alpha^{m-|S|} \E_{X} \left[\prod_{i\in S} f(L_i(X)) \right].
$$
Therefore, a sufficient condition for $A$ to be pseudorandom with regards to $\mathcal{L}$ is that $\E_{X} \left[\prod_{i\in S} f(L_i(X)) \right]$ is negligible for all nonempty subsets $S \subseteq [m]$. Green and Tao~\cite{MR2680398} showed that a sufficient condition for this to occur is that $\|f\|_{U^{s+1}}$ is small enough, where $s$ is the {\em Cauchy-Schwarz complexity} of the system of linear forms.

\begin{definition}[Cauchy-Schwarz complexity~\cite{MR2680398}]
\label{dfn:cauchy-schwarz-complexity}
Let $\mathcal{L}=\{L_1,\ldots,L_m\}$ be a system of linear forms. The {\em Cauchy-Schwarz complexity} of $\mathcal{L}$ is the minimal $s$ such that the following holds. For every $1 \le i \le m$, we can partition $\{L_j\}_{j \in [m] \setminus \{i\}}$ into $s+1$ subsets, such that $L_i$ does not belong to  the linear span of any of the subsets.
\end{definition}
The reason for the term {\em Cauchy-Schwarz complexity} is the following lemma due to Green and Tao~\cite{MR2680398} whose proof is based on a clever iterative application of the Cauchy-Schwarz inequality.

\begin{lemma}[{\cite{MR2680398}, {See also~\cite[Theorem 2.3]{MR2578471}}}]
\label{lem:gowerscount}
Let $f_1,\ldots,f_m:\F \to \D$. Let $\mathcal{L}=\{L_1,\ldots,L_m\}$ be a system of $m$ linear forms in $\ell$ variables of Cauchy-Schwarz complexity $s$. Then
$$
\left| \E_{X \in (\F^n)^{\ell}} \left[ \prod_{i=1}^{m} f_i (L_i(X)) \right]\right| \le \min_{1 \le i \le m} \|f_i\|_{U^{s+1}}.
$$
\end{lemma}
Note that the Cauchy-Schwarz complexity of any system of $m$ linear forms in which any two linear forms are linearly independent (i.e. one is not a multiple of the other) is at most $m-2$, since we can always partition $\{L_j\}_{j \in [m] \setminus \{i\}}$ into the $m-1$ singleton subsets.

The Cauchy-Schwarz complexity of $\mathcal{L}$ gives an upper bound on $s$, such that if $\|f\|_{U^{s+1}}$ is small enough for some function $f:\F^n \rightarrow \mathbb{D}$, then $f$ is pseudorandom with regards to $\mathcal{L}$. Gowers and Wolf~\cite{MR2578471} defined the {\em true complexity} of a system of linear forms as the minimal $s$ such that the above condition holds for all $f:\F^n \rightarrow \mathbb{D}$.
\begin{definition}[True complexity~\cite{MR2578471}]
\label{def:trueComplexity}
Let $\mathcal{L}=\{L_1,\ldots,L_m\}$ be a system of linear forms over $\F$.
The true complexity of $\mathcal{L}$ is the smallest $d \in \N$ with the following property. For every $\eps>0$,
there exists  $\delta>0$ such that if $f : \F^n \rightarrow \D$ is any function
with $\|f\|_{U^{d+1}} \le \delta$, then
$$\left| \E_{X \in (\F^n)^k} \left[ \prod_{i=1}^{m} f (L_i(X)) \right] \right|\le  \eps.$$
\end{definition}
An obvious bound on the true complexity is the Cauchy-Schwarz complexity of the system. However, there are cases where this is not tight. Gowers and Wolf conjectured that the true complexity of a system of linear forms can be characterized by a simple linear algebraic condition. Namely, that it is equal to the smallest $d \ge 1$ such that $L_1^{d+1},\ldots,L_m^{d+1}$ are linearly independent where the $d$-th tensor power of a linear form $L$ is defined as
$$
L^d = \left(\prod_{j=1}^d \lambda_{i_j}: i_1,\ldots,i_d \in [k]\right) \in \F^{k^d}.
$$

Later in~\cite[Theorem 6.1]{MR2773103} they verified their conjecture in the case where $|\F|$ is sufficiently large; more precisely when $|\F|$ is at least the Cauchy-Schwarz complexity of the system of linear form. In this paper we verify the Gowers-Wolf conjecture in full generality by proving the following stronger theorem.

\begin{theorem}\label{thm:gowerswolf}
Let $\mathcal{L}=\{L_1,\ldots,L_m\}$ be a system of linear forms. Assume that $L_1^{d+1}$ is not in the linear span of $L_2^{d+1},\ldots,L_m^{d+1}$. For every $\eps>0$, there exists $\delta>0$ such that for any collection of functions $f_1,\ldots,f_m:\F^n \to \D$ with $\|f_1\|_{U^{d+1}} \le \delta$, we have
$$
\left| \E_{X \in (\F^n)^k} \left[\prod_{i=1}^m f_i(L_i(X)) \right] \right| \le \eps.
$$
\end{theorem}

\Cref{thm:gowerswolf} was conjectured in~\cite{MR2773103}, and left open even in the case of large $|\F|$. In \cite{HL11b}, a partial near-orthogonality result is proved and used to prove \Cref{thm:gowerswolf} in the case where $|\F|$ is greater or equal to the Cauchy-Schwarz complexity of the system of linear form. In this paper, our full near-orthogonality result allows us to establish this theorem in its full generality. We will present the proof of \Cref{thm:gowerswolf} in \Cref{sec:proofGowersWolf}. The following corollary to \Cref{thm:gowerswolf} is very useful when combined with the decomposition theorems such as \Cref{thm:strongdecomposition}.

\begin{corollary}\label{cor:gowerswolf}
Let $\mathcal{L}=\{L_1,\ldots,L_m\}$ be a system of linear forms. Assume that $L_1^{d+1}, \ldots,L_m^{d+1}$ are linearly independent. For every $\eps>0$, there exists $\delta>0$ such that for any functions $f_1,\ldots,f_m, g_1,\ldots, g_m:\F^n \to \D$ with $\|f_i-g_i\|_{U^{d+1}} \le \delta$, we have
$$
\left|\E_{X} \left[\prod_{i=1}^m f_i(L_i(X))\right] - \E_{X} \left[\prod_{i=1}^m g_i(L_i(X))\right]\right| \le \eps,
$$
\end{corollary}
\begin{proof}
Choosing $\delta=\delta(\epsilon')$ as in \Cref{thm:gowerswolf} for $\epsilon':=\epsilon/m$, we have
\begin{align*}
\left| \E_X \left[\prod_{i=1}^m f_i(L_i(X)) \right] -  \E_X \left[\prod_{i=1}^m g_i(L_i(X)) \right]\right|
&=
\left| \sum_{i=1}^m \E_X\left[ (f_i-g_i)(L_i(X)) \cdot \prod_{j=1}^{i-1} g_j(L_j(X))\cdot \prod_{j=i+1}^m
 f_j(L_j(X))\right]\right| \\
&\leq
\sum_{i=1}^m \left|\E_X\left[ (f_i-g_i)(L_i(X)) \cdot \prod_{j=1}^{i-1} g_j(L_j(X))\cdot \prod_{j=i+1}^m
 f_j(L_j(X))\right]\right| \\
&\leq
m\cdot \delta \leq \eps,
\end{align*}
where the second inequality follows from \Cref{thm:gowerswolf} since $\norm{f_i-g_i}_{U^{d+1}}\leq \delta$ and $L_i^{d+1}$ is not in the linear span of $\{L_j^{d+1}\}_{j\in [m]\backslash \{i\}}$.
\end{proof}

\section{Main Proofs}
In this section we will present the proofs of \Cref{thm:homogeneous-basis}, \Cref{thm:nearorthogonality}, and \Cref{thm:gowerswolf}.

\subsection{Homogeneity: Proof of \Cref{thm:homogeneous-basis}}\label{sec:proofhomogeneous-basis}
\restate{\Cref{thm:homogeneous-basis}}{
There is a basis for $\poly(\F^n\rightarrow \T)$ consisting only of homogeneous multivariate polynomials.}

To simplify the notation, in this section we will omit writing ``$\mathrm{mod}\; 1$'' in the description of the defined non-classical polynomials. We start by proving the following simple observation.

\begin{claim}\label{claim:multiplicativeconstant}
Let $P:\F \rightarrow \T$ be a univariate polynomial of degree $d$. Then for every $c\in \F\backslash \{0\}$,
$$
\deg\left(P(cx)-|c|^d  P(x) \right) < d.
$$
\end{claim}
\begin{proof}
 By \cref{struct} it suffices to prove the claim for a monomial $q(x):=\frac{|x|^{s}}{p^{k+1}}$ with $k(p-1) + s = d$. Note that $q(cx)-|c|^dq(x)$ takes values in $\frac{1}{p^k} \mathbb{Z}/\mathbb{Z}$ as $|c|^{s}-|c|^{d}$ is divisible by $p$. Hence
\begin{equation}
\label{eq:degreedrop}
\deg\left(q (cx)-|c|^d  q (x) \right) \le (p-1)(k-1)<d.
\end{equation}
\end{proof}

\begin{remark}
It is not  difficult to show that the above claim holds for any multivariate polynomial $P:\F^n \rightarrow \T$. However, since the univariate case suffices for our purpose, we do not prove the general case.
\end{remark}

\ignore{
\begin{proof}
Notice that the claim is trivial for classical polynomials, since in this case, if $R$ denotes the homogeneous degree-$d$ part of $P$, then $R(cx)-|c|^dR(x)=0$. We prove the statement for non-classical polynomials. Let $Q(x) := P(cx)$, and note $\deg(Q)=d$. We will inspect the derivative polynomial of $Q$. Recall from \Cref{dfn:derivativepolynomial} and (\Cref{lem:derivativepolynomial} that the derivative polynomial of $Q$,
$$
\partial Q(y_1,\ldots,y_d)= D_{y_1}\cdots D_{y_d}Q(0),
$$
is a degree-$d$ classical homogeneous multi-linear polynomial which is invariant under permutations of $(y_1,\ldots,y_d)$. In particular
\begin{align*}
|c|^{-d}\partial Q(y_1,\ldots, y_d)&= \partial Q(c^{-1}y_1,\ldots,c^{-1}y_d) = D_{c^{-1}y_1} D_{c^{-1}y_2}\cdots D_{c^{-1}y_d}Q(x)\\ &=
\sum_{S\subseteq [d]} (-1)^{|S|} Q\left(c^{-1}\sum_{i\in S} y_i\right) =
\sum_{S\subseteq [d]} (-1)^{|S|} P\left(\sum_{i\in S} y_i\right) \\ &= \partial P(y_1,\ldots, y_d),
\end{align*}
This implies that $\partial_d(Q- |c|^dP)\equiv 0$ and thus $\deg(Q-|c|^dP) <d$.
\end{proof}
}

First we prove \Cref{thm:homogeneous-basis} for univariate polynomials.

\begin{lemma}\label{lem:homogeneous-basis-univariate}
There is a basis of homogeneous univariate polynomials for $\poly(\F\rightarrow \T)$.
\end{lemma}
\begin{proof}
We will prove by induction on $d$ that there is a basis $\{h_1,\ldots,h_d\}$ of homogeneous univariate polynomials  for $\poly_{\leq d}(\F\rightarrow \T)$ for every $d$. Let $\zeta$ be a fixed generator of $\F^*$. For any degree $d>0$, we will build a degree-$d$ homogeneous polynomial $h_d(x)$ such that $h_d(\zeta x)=\sigma_d  h_d(x)$ for some integer $\sigma_d$. The base case of $d\leq p-1$ is trivial as $\poly_{\leq p-1}(\F\rightarrow \T)$ consists of only classical polynomials, and those are spanned by $h_0(x):=\frac{1}{p},h_1(x):=\frac{|x|}{p}, \ldots,h_{p-1}(x):=\frac{|x|^{p-1}}{p}$. Now suppose that $d=s+(p-1)(k-1)$ with $0<s\leq p-1$, and $k>1$. It suffices to show that the degree-$d$ monomial $\frac{|x|^s}{p^k}$ can be expressed as a linear combination of homogeneous polynomials. Consider the function
$$
f(x) := \frac{|\zeta x|^s}{p^k} - \frac{|\zeta|^s |x|^s}{p^k}.
$$
\Cref{claim:multiplicativeconstant} implies that $\deg(f)< d$. Using the induction hypothesis, we can express $f(x)$ as a linear combination of $\frac{|x|^{s}}{p^\ell}$ for $\ell=0,\ldots,k-1$, and $h_e$ for $e<d$ with $e \neq s \mod (p-1)$:
$$
f(x)= \sum_{\ell=1}^{k-1} a_{\ell} \frac{|x|^s}{p^\ell} + \sum_{\substack{e<d,\\ e \ne d \mod (p-1)}} b_e h_e(x).
$$
Set $A:=|\zeta|^s + \sum_{\ell=1}^{k-1} a_\ell p^{k-\ell}$, so that
\begin{equation}
\label{eq:claimuni}
\frac{|\zeta x|^s}{p^k} - A \frac{|x|^s}{p^k} = \sum_{\substack{e<d,\\ e \ne d \mod (p-1)}} b_e h_e(x).
\end{equation}
By the induction hypothesis, for $e<d$, $h_e(\zeta x)=\sigma_e h(x)$ where $\sigma_e = |\zeta|^e \mod p$, and thus as $A = |\zeta |^s \mod p$, we have $\sigma_e \neq A \mod p$ when $e \neq s \mod (p-1)$. Consequently,
\begin{align*}
\sum_{\substack{e<d,\\ e \ne d \mod (p-1)}} b_e h_e(x)
&= \sum_{\substack{e<d,\\ e \ne d \mod (p-1)}} \frac{b_e}{\sigma_e-A} (\sigma_e - A) h_e(x)
\\ &= \sum_{\substack{e<d,\\ e \ne d \mod (p-1)}} \frac{b_e}{\sigma_e-A} (h_e(\zeta x) - A h_e(x)).
\end{align*}
Combing this with \eqref{eq:claimuni} we conclude  that
$$
h_d(x) := \frac{|x|^s}{p^k} -  \sum_{\substack{e<d,\\ e \ne d \mod (p-1)}}  \frac{b_e}{\sigma_e-A} h_e(x),
$$
satisfies
$$
h_d(\zeta x) = A h_d(x).$$
\ignore{ Finally to show  $A=|c|^d \mod p$, note that by \Cref{claim:multiplicativeconstant} we have $h_d(cx)-|c|^dh_d(x)=(A-|c|^d)h_d(x)$ is of degree $<d$ and thus $A-|c|^d$ must be divisible by $p$.}
\end{proof}

\begin{proofof}{of \Cref{thm:homogeneous-basis}}
We will show by induction on the degree $d$, that every degree $d$ monomial can be written as a linear combination of homogeneous polynomials. The base case of $d<p$ is trivial as such monomials are classical and thus homogeneous themselves. Consider a (non-classical) monomial $M(x_1,\ldots,x_n)=\frac{|x_1|^{s_1} \cdots |x_n|^{s_n}}{ p^k}$ of degree $d=s_1+\cdots+s_n+(p-1)(k-1)$. For every $i\in [n]$ let $g_i(x_i):=h_{s_i+(p-1)(k-1)}(x_i)$ where $h_{s_i+(p-1)(k-1)}(\cdot)$ is the homogeneous univariate polynomial from \Cref{lem:homogeneous-basis-univariate}. Every $g_i$ takes values in $\frac{1}{p^k}\Z/\Z$, and thus corresponds to a polynomial $G_i:\F \to \Z_{p^k}$. Define $F:\F^n \rightarrow \Z_{p^k}$ as
$$
F(x_1,\ldots,x_n) := G_1(x_1) \cdots G_n(x_n),
$$
and $f:\F^n\rightarrow \T$ as
$$
f(x_1,\ldots,x_n) := \frac{F(x_1,\ldots,x_n)}{p^k}.
$$
It is simple to verify that $\deg(f) = s_1+\ldots+s_n + (p-1)(k-1)=d$, it has only one monomial of $\deg(f)$, which is $\frac{|x_1|^{s_1}\cdots |x_n|^{s_n}}{ p^k}=M(x_1,\ldots,x_n)$, and it is homogeneous. Thus $M(x_1,\ldots,x_n)-f(x_1,\ldots,x_n)$ is of degree less than $d$ and by the induction hypothesis can be written as a linear combination of homogeneous polynomials.
\end{proofof}

\subsection{Near-orthogonality: Proof of \Cref{thm:nearorthogonality}}
\label{sec:proofofnearorthogonality}

\restate{\Cref{thm:nearorthogonality}}{
Let $L_1,\ldots, L_m$ be linear forms on $\ell$ variables and let $\cB=(P_1,\ldots,P_C)$ be an $\epsilon$-uniform polynomial factor for some $\epsilon\in (0,1]$ defined only by homogeneous polynomials. For every tuple $\Lambda$ of integers $(\lambda_{i,j})_{i\in [C], j\in [m]}$, define $P_\Lambda:(\F^n)^\ell\rightarrow \T$ as
$$
P_{\Lambda}(X)= \sum_{i\in [C], j\in [m]} \lambda_{i,j} P_i(L_j(X)).
$$
Then one of the following two statements holds: 
\begin{itemize}
\item $P_\Lambda \equiv 0.$
\item $P_\Lambda$ is non-constant and $\left| \E_{X\in (\F^n)^\ell} [\expo{P_\Lambda}] \right|< \epsilon$.
\end{itemize}
Furthermore $P_\Lambda \equiv 0$ if and only if for every $i \in [C]$, we have  $(\lambda_{i,j})_{j \in [m]} \in \Phi_{d_i,k_i}(\cL)^\perp$ where $d_i,k_i$ are the degree and depth of $P_i$, respectively.  }

We prove~\Cref{thm:nearorthogonality} in this section. Our proof uses similar derivative techniques as used in \cite{BFHHL13}, but in order to handle the general setting we will need a few technical claims which we present first. Recall that $|L|=\sum_{i=1}^\ell |\lambda_i|$ for a linear form $L=(\lambda_1,\ldots,\lambda_\ell)$.

\begin{claim}\label{claim:linearreform}
Let $d>0$ be an integer, and $L=(\lambda_1,\ldots,\lambda_\ell)\in \F^\ell$ be a linear form on $\ell$ variables. There exists linear forms $L_i=(\lambda_{i,1},\ldots,\lambda_{i,\ell})\in \F^\ell$ for $i=1,\ldots,m$, and coefficients $a_1,\ldots,a_m\in \Z$ with $m\leq |\F|^\ell$ such that
\begin{itemize}
\item $P(L(X))= \sum_{i=1}^m a_i P(L_i(X))$ for every degree-$d$ polynomial $P:\F^n\rightarrow \T$;
\item $|L_i|\leq d$ for every $i \in [m]$;
\item $\lambda_{i,j}\leq \lambda_j$ for every $i \in [m]$ and $j\in [\ell]$.
\end{itemize}
\end{claim}
\begin{proof}
The proof proceeds by simplifying $P(L(X))$ using identities that are valid for every polynomial  $P:\F^n\rightarrow \T$ of degree $d$.

In the case  $|L|\leq d$ there is nothing to prove. Assume otherwise that $|L|>d$. We will use the fact that for every choice of $y_1,\ldots,y_{|L|}\in \F^n$,
\begin{equation}\label{eq:reform1}
\sum_{S\subseteq [|L|]} (-1)^{|S|} P\left(\sum_{i\in S}y_i\right)\equiv 0.
\end{equation}
Let $X=(x_1,\ldots,x_\ell)\in (\F^n)^\ell$. Setting $|\lambda_i|$ of the vectors $y_1,\ldots,y_{|L|}$ to $x_i$ for every $i \in [\ell]$,  \Cref{eq:reform1} implies
$$
P(L(X))= \sum_i \alpha_i P(M_i(X)),
$$
where $M_i=(\tau_{i,1},\ldots, \tau_{i,\ell})$, $|M_i|\leq |L|-1$ and for every $j\in [\ell]$, $|\tau_{i,j}| \leq |\lambda_j|$. Repeatedly applying the same process to every $M_i$ with $|M_i|>d$ we arrive at the desired expansion.
\end{proof}

The next claim shows that we can further simplify the expression given in \cref{claim:linearreform}. Let $\LL_{d} \subseteq \F^\ell$ denote  the set of nonzero linear forms $L$ with $|L|\leq d$ and with the first (left-most) nonzero coefficient equal to $1$, e.g. $(0,1,0,2)\in \LL_{3}$ but $(2,1,0,0)\not\in \LL_{3}$.

\begin{claim}\label{claim:lc1}
For any linear form $L\in \F^\ell$ and integer $d>0$, there is a collection of  coefficients $\{a_{M,c}\in \Z\}_{M \in \LL_d, c\in \F^*}$ such that
for every degree-$d$ polynomial $P:\F^n\rightarrow \T$,
\begin{equation}\label{eq:claimlc1}
P(L(X))= \sum_{M \in \LL_d, c \in \F^*} a_{M,c}  P(c M(X)).
\end{equation}
\end{claim}
\begin{proof}
Similar to the proof of \Cref{claim:linearreform} we simplify $P(L(X))$ using identities that are valid for every polynomial  $P:\F^n\rightarrow \T$ of degree $d$.

We use induction on the number of nonzero entries of $L$. The case when $L$ has only one nonzero entry is trivial. For the induction step, choose $c\in \F$ so that the leading nonzero coefficient of $L'=c\cdot L$ is equal to $1$. Assume that $L'=(\lambda_1,\ldots,\lambda_\ell)$. If $|L'|\leq d$ we are done. Assume otherwise that $|L'|>d$. Applying \Cref{claim:linearreform} for the degree-$d$ polynomial $R(x):= P(c^{-1} x)$ and the linear form $L'$ we can write
\begin{equation}\label{eq:lc1-2}
P(L(X))=P(c^{-1} L'(X))= \sum_i \beta_i P(c^{-1} M_i(X)),
\end{equation}
where for every $i$, $M_i=(\lambda_{i,1},\ldots, \lambda_{i,\ell})$ satisfies $|M_i|\leq d$, and for every $j\in [\ell]$, $\lambda_{i,j}\leq \lambda_j$.
Let $\cI$ denote the set of indices $i$ such that the leading nonzero entry of $M_i$ is one. Then
\begin{equation}\label{eq:lc1-3}
P(L(x))= \sum_{i\in \cI} \alpha_i P(c^{-1} M_i(X)) + \sum_{j\notin \cI} \alpha_j P(c^{-1} M_j(X)).
\end{equation}
Notice that since the leading coefficient of $L'$ is $1$, for every $j\notin \cI$, $M_j$ has smaller support than $L$ and thus applying the induction hypothesis to the linear forms $c^{-1} M_j$ with $j\notin \cI$ concludes the claim.
\end{proof}

\ignore{
\begin{claim}\label{claim:lc12}
For any linear form $L\in \F^\ell$, degree $d \ge 1$ and depth $0 \le k \le \left \lfloor \frac{d-1}{p-1}\right \rfloor$, there exist sets of coefficients $\{a_{M,c} \in \Z\}_{M \in \LL_d, c \in \F^*}$ such that the following is true for every polynomial $P:\F^n\rightarrow \T$ of degree $d$ and depth $k$:
\begin{itemize}
\item $ P(L(X))= \sum_{M \in \LL_d, c \in \F^*} a_{M,c}  P(c M(X))$;
\item For every $M$ and $c$ with $a_{M,c} \neq 0$, we have $|M| \le \deg(a_{M,c} P)$.
\end{itemize}
\end{claim}
\begin{proof}
The proof is similar to that of \Cref{claim:lc1}, except that we will repeatedly apply \Cref{claim:linearreform} to the polynomials $a_{M,c} P$ in order to achieve $|M|\leq \deg(a_{M,c} P)$. Notice that  $\deg(a_{M,c} P)$ depends only on the depth and the degree of $P$ and not on the particular choice of the polynomial.
\end{proof}}

\Cref{claim:lc1} applies to all polynomials of degree $d$. If we also specify the depth then we can obtain a stronger statement.

\begin{claim}\label{claim:lc12}
For every $d,k$,  every system of linear forms $\{L_1,\ldots,L_m\}$, and constants $\{\lambda_i \in \mathbb{Z}\}_{i \in [m]}$, there exists  $\{a_{M} \in \Z\}_{M \in \LL_d}$ such that the following is true for every $(d,k)$-homogeneous polynomial $P:\F^n \to \T$:
\begin{itemize}
\item $ \sum_{i=1}^m \lambda_i P(L_i(X)) \equiv \sum_{M \in \LL_d} a_{M}  P(M(X))$;
\item For every $M$ with $a_{M} \neq 0$, we have $|M| \le \deg(a_{M} P)$.
\end{itemize}
\end{claim}
\begin{proof}
The proof is similar to that of \Cref{claim:lc1}, except that now we repeatedly apply \Cref{claim:lc1} to every term of the form $\lambda P(L(X))$ to express it as a linear combination of $P(cM(X))$ for $M \in \LL_{\deg(\lambda P)}$ and $c \in \F^*$. Then we use homogeneity to replace $P(cM(X))$ with $\sigma_c P(M(X))$, where if $c=\zeta^i$ for the fixed generator $\zeta \in \F^*$ then $\sigma_c=\sigma(d,k)^i$. By repeating this procedure we arrive at the desired expansion.
\end{proof}

We are now ready for the proof of our main theorem.  For a linear form $L=(\lambda_1,\ldots,\lambda_\ell)$, let $\lc(L)$ denote the index of its first nonzero entry, namely $\lc(L)\defeq \min_{i:\lambda_i\neq 0} i$.

\begin{proofof}{of \Cref{thm:nearorthogonality}}
Let $d'$ be the degree of the factor. For every $i \in [C]$, by   \Cref{claim:lc12} we have
\begin{equation}
\label{eq:nearorthog1}
\sum_{j=1}^m \lambda_{i,j}P_i(L_j(X)) = \sum_{M \in \LL_{d'}} \lambda'_{i,M} P_i(M(X))
\end{equation}
for some integers  $\lambda'_{i,M}$ such that $|M|\leq \deg(\lambda'_{i,M}P_i)$ if $\lambda'_{i,M} \neq 0$. The simplifications of  \Cref{claim:lc12} depend only on the degrees and depths of the polynomials. Hence if $\lambda'_{i,M}=0$ for all $M \in \LL_{d'}$, then $(\lambda_{i,1},\ldots,\lambda_{i,m}) \in \Phi(d_i,k_i)^\perp$. So to prove the theorem, it suffices to show that $P_\Lambda$ has small bias if $\lambda'_{i,M} P_i \not \equiv 0$ for some  $M \in \LL_{d'}$. Suppose this is true, and thus there exists a nonempty set $\cM \subseteq \LL_{d'}$ such that
$$
P_\Lambda(X)= \sum_{i\in [C],M \in \cM} \lambda'_{i,M} P_i(M(X)),
$$
and for every $M \in \cM$, there is at least one index $i \in [C]$ for which $\lambda'_{i,M} P_i \ne 0$. Choose $i^*\in [C]$ and $M^* \in \cM$ in the following manner.
\begin{itemize}
\item First, let $M^* \in \cM$ be such that  $\lc(M^*)=\min_{M \in \cM} \lc(M)$, and among these, $|M^*|$ is maximal.
\item Then, let $i^*\in [C]$ be such that $\deg(\lambda'_{i^*,M^*}P_{i^*})$ is maximized.
\end{itemize}

Without loss of generality assume that $i^*=1$, $\lc(M^*)=1$, and let $d:=\deg(\lambda'_{1,M^*}P_1)$.
We claim that if $\sum_{j\in [m]} \lambda_{1,j}P_1(L_j(X))$ is not the zero polynomial, then $\deg(P_\Lambda) \geq d$, and moreover $P_\Lambda$ has small bias. We prove this by deriving $P_\Lambda$ in specific directions in a manner that all the terms but $ \lambda'_{1,M^*}P_1(M^*(X))$ vanish.

Given a vector $\alpha\in \F^\ell$, an element $y\in \F^n$, and a function $P:(\F^n)^\ell\rightarrow \T$, define the derivative of $P$ according to the pair $(\alpha,y)$ as
\begin{equation}
D_{\alpha,y}P(x_1,\ldots,x_\ell)\defeq P(x_1+\alpha_1y, \cdots, x_\ell+\alpha_\ell y) - P(x_1,\ldots, x_\ell).
\end{equation}
Note that for every $M \in \cM$,
\begin{align*}
D_{\alpha,y}(P_i\circ M)(x_1,\cdots,x_\ell)&= P_i(M(x_1,\ldots, x_\ell)+M(\alpha)y)- P_i(M(x_1,\ldots, x_\ell))\\ &= (D_{\langle M, \alpha \rangle \cdot y}P_i)(M(x_1,\ldots, x_\ell)).
\end{align*}
Thus if $\alpha$ is chosen such that $\langle M,\alpha\rangle = 0$ then $D_{\alpha,y}(P_i\circ M) \equiv 0$.

Assume that $M^*=(w_{1},\ldots, w_{\ell})$, where $w_{1}=1$. Let $t:=|M^*|$, $\alpha_1:=e_{1}=(1,0,0,\ldots,0)\in \F^\ell$, and let $\alpha_2,\ldots, \alpha_t$ be the set of all vectors of the form $(-w,0,\ldots,0,1,0,\ldots,0)$ where $1$ is in the $i$-th coordinate for $i\in [2,\ell]$ and $0\leq w \leq w_{i}-1$. In addition, pick $\alpha_{t+1}=\cdots=\alpha_{d}=e_1$.
\begin{claim}\label{claim:deriving}
\begin{equation}\label{eq:finalroundofderivatives}
D_{\alpha_1,y_1}\cdots D_{\alpha_d,y_d} P_\Lambda(X) =
\bigg(D_{\langle M^*, \alpha_1\rangle y_1} \cdots D_{\langle M^*, \alpha_d \rangle y_d} \sum_{\substack{i\in [C]: \\ \deg(\lambda'_{i,M^*}P_i)=d}} \lambda'_{i,1}P_i\bigg)(M^*(X)).
\end{equation}
\end{claim}
\begin{proof}
Deriving according to $(\alpha_1,y_1),\ldots,(\alpha_{t},y_{t})$ gives
\begin{equation}\label{eq:firstroundofderivatives}
D_{\alpha_1,y_1}\cdots D_{\alpha_{t},y_{t}} P_\Lambda(X) =
\left(D_{\langle M^*, \alpha_1\rangle y_1} \cdots D_{\langle M^*, \alpha_t\rangle y_t} \left(\sum_{i=1}^{C} \lambda'_{i,M^*}P_i\right)\right)(M^*(X)).
\end{equation}
This is because for every $M=(w'_1,\ldots,w'_\ell) \in \cM \setminus \{M^*\}$, either $w'_{1}=0$ in which case $\langle M, \alpha_1\rangle =0$, or otherwise $w_{1}=1$ and $|M|\leq t$, thus there must be an index $\xi \in [\ell]$ such that $w'_{\xi}<w_{\xi}$; By our choice of $\alpha_2,\ldots, \alpha_{t}$ there is $e$, $2\leq e \leq t$, such that $\alpha_e = (-w'_{\xi},0,\ldots,0,1,0,\ldots,0)$ where $1$ is in the $\xi$-th coordinate and thus $\langle M_j, \alpha_e \rangle =0$. Now, the claim follows after additionally deriving according to $(\alpha_{t+1},y_{t+1}),\ldots, (\alpha_d,y_d)$.
\end{proof}

\Cref{claim:deriving} implies that
$$
\E_{\substack{y_1,\ldots,y_d,\\ x_1,\ldots, x_\ell}} \left[ \expo{D_{\alpha_1,y_1}\cdots D_{\alpha_d,y_d}P_\Lambda)(x_1,\ldots,x_\ell)} \right] = \left\|\expo{\sum_{\substack{i\in[C]:\\ \deg(\lambda'_{i,1}P_i) = d}} \lambda'_{i,M^*}P_i}\right\|_{U^d}^{2^d}\leq \epsilon^{2^d},
$$
where the last inequality holds by the $\epsilon$-uniformity of the polynomial factor. Now the theorem follows from the next claim from \cite{BFHHL13} which is a repeated application of the Cauchy-Schwarz inequality. We include a proof for self-containment.

\begin{claim}[{{\cite[Claim 3.4]{BFHHL13}}}]
For any $\alpha_1,\ldots, \alpha_d\in \F^\ell\backslash \{\underline{0}\}$,
\begin{align*}
\E_{\substack{y_1,\dots, y_d,\\ x_1, \dots,
 x_\ell}} \left[ \expo{(D_{\alpha_1,y_1}\cdots D_{\alpha_d,y_d}P_\Lambda)(x_1,\ldots,x_\ell)} \right] \geq \left(\left|\E_{x_1,\dots,x_\ell} \expo{P_{\Lambda}(x_1,\dots,x_\ell)} \right|\right)^{2^d}.
\end{align*}
\end{claim}
\begin{proof}
It suffices to show that for any function $P(x_1,\dots,x_\ell)$ and
nonzero $\bm{\alpha} \in \F^\ell$,
$$\left|\E_{y, x_1, \dots, x_\ell}[\expo{(D_{\bm{\alpha},
 y}P)(x_1,\dots,x_\ell)}]\right| \geq
\left|\E_{x_1, \dots, x_\ell}[\expo{P(x_1,\dots,x_\ell)}]\right|^2.
$$
Recall that $(D_{\bm{\alpha},
 y}P)(x_1,\dots,x_\ell) = P(x_1 + \alpha_1 y, \dots, x_\ell + \alpha_\ell y)
 - P(x_1, \dots, x_\ell)$. Without loss of generality, suppose $\alpha_1
 \neq 0$. We make a change of coordinates so that
 $\bm{\alpha}$ can be assumed to be $(1,0,\dots, 0)$. More precisely,
 define $P': (\F^n)^\ell \to \T$ as
$$P'(x_1, \dots, x_\ell) = P\left(x_1,
 \frac{x_2 + \alpha_2 x_1}{\alpha_1}, \frac{x_3 + \alpha_3 x_1}{\alpha_1}, \dots,
 \frac{x_\ell + \alpha_\ell x_1}{\alpha_1}\right),$$
so that $P(x_1, \dots, x_\ell) = P'(x_1, \alpha_1 x_2 - \alpha_2 x_1, \alpha_1 x_3 - \alpha_3x_1, \dots,
 \alpha_1 x_\ell - \alpha_\ell x_1)$, and thus $(D_{\bm{\alpha},
 y}P)(x_1,\dots,x_\ell) = P'(x_1 + \alpha_1 y, \alpha_1 x_2 - \alpha_2
 x_1, \dots, \alpha_1 x_\ell -\alpha_\ell x_1) - P'(x_1, \alpha_1 x_2 - \alpha_2
 x_1, \dots, \alpha_1 x_\ell -\alpha_\ell x_1)$. Therefore
\begin{align*}
&\left|\E_{y, x_1, \dots, x_\ell}[\expo{(D_{\bm{\alpha},
 y}P)(x_1,\dots,x_\ell)}]\right|\\
&= \left|\E_{y, x_1, \dots, x_\ell}[\expo{P'(x_1 + \alpha_1 y, \alpha_1 x_2 - \alpha_2
 x_1, \dots, \alpha_1 x_\ell -\alpha_\ell x_1) - P'(x_1, \alpha_1 x_2 - \alpha_2
 x_1, \dots, \alpha_1 x_\ell -\alpha_\ell x_1)}]\right|\\
&= \left|\E_{y, x_1, \dots, x_\ell}[\expo{P'(x_1 + \alpha_1 y, x_2,
 \dots, x_\ell) - P'(x_1, x_2, \dots, x_\ell)}]\right|= \E_{x_2,\dots,x_\ell}\left|\E_{x_1}[\expo{P'(x_1, x_2, \dots,
 x_\ell)}]\right|^2\\
&\geq \left|\E_{x_1, x_2, \dots, x_\ell} [\expo{P'(x_1, x_2, \dots,
 x_\ell)}]\right|^2
= \left|\E_{x_1, x_2, \dots, x_\ell} [\expo{P(x_1, x_2, \dots,
 x_\ell)}]\right|^2.
\end{align*}
\end{proof}
\end{proofof}

The above proof also implies the following proposition just by omitting the application of \Cref{claim:lc12}.

\begin{proposition}\label{pro:nearorthog-niceform}
Let $L_1,\ldots, L_m$ be linear forms on $\ell$ variables and let $\cB=(P_1,\ldots,P_C)$ be an $\epsilon$-uniform polynomial factor of degree $d>0$ for some $\epsilon\in (0,1]$ which is defined by only homogeneous polynomials. For every tuple $\Lambda$ of integers $(\lambda_{i,j})_{i\in [C], j\in [m]}$, define
$$
P_{\Lambda}(X)= \sum_{i\in [C], j\in [m]} \lambda_{i,j} P_i(L_j(X)),
$$
where $P_\Lambda:(\F^n)^\ell\rightarrow \T$. Moreover assume that for every $i\in [C],j\in [m]$, $L_j\in \LL_{\deg(\lambda_{i,j}P_i)}$. Then, $P_\Lambda$ is of degree $d=\max_{i,j}\deg(\lambda_{i,j}P_i)$ and $\norm{ \expo{P_\Lambda}}_{U^d}< \epsilon$.
\end{proposition}

\subsection{The Gowers-Wolf Conjecture: Proof of \Cref{thm:gowerswolf}}
\label{sec:proofGowersWolf}

\restate{\Cref{thm:gowerswolf}}{
Let $\mathcal{L}=\{L_1,\ldots,L_m\}$ be a system of linear forms. Assume that $L_1^{d+1}$ is not in the linear span of $L_2^{d+1},\ldots,L_m^{d+1}$. For every $\eps>0$, there exists $\delta>0$ such that for any collection of functions $f_1,\ldots,f_m:\F^n \to \D$ with $\|f_1\|_{U^{d+1}} \le \delta$, we have
\begin{equation}
\label{eq:mainGowerWolf}
\left| \E_{X \in (\F^n)^k} \left[\prod_{i=1}^m f_i(L_i(X)) \right] \right| \le \eps.
\end{equation}
}

We prove~\Cref{thm:gowerswolf} in this section.
Note that since $L_1^{d+1}$ is not in the linear span of $L_2^{d+1},\ldots, L_m^{d+1}$ we have that $L_1$ is linearly independent from each $L_j$ for $j>1$. We claim that we may assume without loss of generality that $L_2,\ldots,L_m$ are pairwise linearly independent as well, namely that $\{L_1,\ldots,L_m\}$ has bounded Cauchy-Schwarz complexity. Assume that there are $j,\ell\in [m]\backslash\{1\}$ and a nonzero $c\in \F$ such that $L_{\ell} =cL_{j}$. Then we may define a new function $f'_j(x)= f_j(x)f_{\ell} (cx)$ so that $f'_j(L_j(X))= f_j(L_j(X))f_{\ell} (L_{\ell} (X))$ and remove the linear form $L_{\ell} $ and functions $f_j,f_{\ell} $ from the system. Now
$$
\E_{X \in (\F^n)^k} \left[\prod_{i=1}^m f_i(L_i(X)) \right] =
\E_{X \in (\F^n)^k} f'_j(L_j(X))\cdot \left[\prod_{i\in [m]\backslash \{j,\ell\}} f_i(L_i(X)) \right],
$$
and thus it suffices to bound the right hand side of the above identity. We may repeatedly apply the above procedure in order to achieve a new system of pairwise linearly independent linear forms along with their corresponding functions while keeping $L_1$ and $f_1$ untouched.

Thus we may assume that $\{L_1,\ldots,L_m\}$ is of finite Cauchy-Schwarz complexity $s$ for some $s \leq m < \infty$. The case when $d\geq s$ follows from \Cref{lem:gowerscount}, thus we will consider the case when $d<s$. We will use \Cref{cor:homogeneous-strong-decomposition} to write $f_i=g_i+h_i$ with
\begin{enumerate}
\item $g_i= \E[f_i|\cB]$, where $\cB$ is an $r$-regular polynomial factor of degree at most $s$ and complexity $C\leq C_{\mathrm{\mathrm{max}}}(p, s, \eta,\delta, r(\cdot))$ defined by only homogeneous polynomials, where $r$ is a sufficiently fast growing growth function (to be determined later);
\item $\norm{h_i}_{U^{s+1}}\leq \eta$.
\end{enumerate}

We first show that by choosing a sufficiently small $\eta$ we may replace $f_i$'s in \Cref{eq:mainGowerWolf} with $g_i$'s.
\begin{claim}
Choosing $\eta\leq \frac{\eps}{2m}$ we have
$$
\left| \E_{X \in (\F^n)^k} \left[\prod_{i=1}^m f_i(L_i(X)) \right] - \E_{X \in (\F^n)^k} \left[\prod_{i=1}^m g_i(L_i(X)) \right] \right| \leq \frac{\epsilon}{2}.
$$
\end{claim}
\begin{proof}
We have
\begin{align*}
\left| \E_X \left[\prod_{i=1}^m f_i(L_i(X)) \right] -  \E_X \left[\prod_{i=1}^m g_i(L_i(X)) \right]\right| &= \left| \sum_{i=1}^m \E_X\left[ h_i(L_i(X)) \cdot \prod_{j=1}^{i-1} g_j(L_j(X))\cdot \prod_{j=i+1}^m
 f_j(L_j(X))\right]\right| \\ &\leq
\sum_{i=1}^m \left|\E_X\left[ h_i(L_i(X)) \cdot \prod_{j=1}^{i-1} g_j(L_j(X))\cdot \prod_{j=i+1}^m
 f_j(L_j(X))\right]\right| \\ &\leq
\sum_{i=1}^m \norm{h_i}_{U^{s+1}} \leq m\cdot \eta \leq \frac{\eps}{2},
\end{align*}
where the second inequality follows from \Cref{lem:gowerscount} since the Cauchy-Schwarz complexity of $\L$ is $s$.
\end{proof}

Thus it is sufficient to bound $\left| \E_{X \in (\F^n)^k} \big[\prod_{i=1}^m g_i(L_i(X)) \big] \right|$ by $\eps/2$. For each $i$, $g_i= \E[f_i|\cB]$ and thus
$$
g_i(x)= \Gamma_i(P_1(x),\ldots, P_C(x)),
$$
where $P_1,\ldots, P_C$ are the (non-classical) homogeneous polynomials of degree $\leq s$ defining $\cB$ and $\Gamma_i:\T^C\rightarrow \D$ is a function. Let $k_i$ denote the depth of the polynomial $P_i$ so that by \Cref{struct}, each $P_i$ takes values in $\U_{k_i+1}=\frac{1}{p^{k_i+1}} \Z/\Z$. Moreover let $\Sigma:=\Z_{p^{k_1+1}}\times \cdots \times \Z_{p^{k_C+1}}$. Using the Fourier expansion of $\Gamma_i$ we have
\begin{equation}\label{eq:mainsimplefourier}
g_i(x)= \sum_{\Lambda=(\lambda_1,\ldots, \lambda_C)\in \Sigma} \widehat{\Gamma}_i(\Lambda)\cdot \expo{\sum_{j=1}^C \lambda_j P_j(x)},
\end{equation}
where $\widehat{\Gamma}_i(\Lambda)$  is the Fourier coefficient of $\Gamma_i$ corresponding to $\Lambda$. Let $P_\Lambda:=\sum_{j=1}^C \lambda_j P_j(x)$ for the sake of brevity so that we may write
\begin{equation}\label{eq:mainfourier}
\E_X\left[\prod_{i=1}^m g_i(L_i(X))\right] = \sum_{\Lambda_1,\ldots,\Lambda_m\in \Sigma} \left( \prod_{i=1}^m \widehat{\Gamma}_i(\Lambda_i)\right) \cdot \E_X \left[\expo{\sum_{i=1}^m P_{\Lambda_i}(L_i(X))}\right].
\end{equation}
We will show that for a sufficiently fast growing choice of the regularity function $r(\cdot)$ we may bound each term in \Cref{eq:mainfourier} by $\sigma:=\frac{\eps}{2|\Sigma|^m}$, thus concluding the proof by the triangle inequality. We will first show that the terms for which $\deg(P_{\Lambda_1})\leq d$ can be made small.
\begin{claim}
Let $\Lambda\in \Sigma$ be such that $\deg(P_\Lambda)\leq d$. For a sufficiently fast growing choice of $r(\cdot)$ and choice of $\delta\leq \frac{\sigma}{2}$ ,
$$
\left|\widehat{\Gamma}_1(\Lambda)\right|<\sigma.
$$
\end{claim}
\begin{proof}
It follows from \Cref{eq:mainsimplefourier} that
$$
\widehat{\Gamma}_1(\Lambda) = \E_x \left[g_1(x) \expo{-P_\Lambda(x)}\right] - \sum_{\Lambda' \in \Sigma \backslash \{\Lambda\}} \widehat{\Gamma}_1(\Lambda')\cdot \E_x\left[\expo{P_{\Lambda'}(x)-P_\Lambda(x)}\right].
$$
Note that $\norm{f_1}_{U^{d+1}} \le \delta$ and thus$$
 \left|\E_x \left[g_1(x) \expo{-P_\Lambda(x)}\right]\right| =\left|  \E_x \left[f_1(x) \expo{-P_\Lambda(x)}\right]\right| \leq \norm{f_1  \expo{-P_\Lambda}}_{U^{d+1}}=  \norm{f_1}_{U^{d+1}} \leq \delta\leq \sigma/2,
$$
where we used of the fact that $g_1=\E[f_1|\cB]$ and the fact that Gowers norms are increasing in $d$. Finally the terms of the form $\widehat{\Gamma}_1(\Lambda')\cdot \E_x[\expo{P_{\Lambda'}(x)-P_\Lambda(x)}]$ with $\Gamma'\neq \Gamma$ can be made arbitrarily small by choosing a sufficiently fast growing $r(\cdot)$ due to \Cref{remark:uniformityvsrank}, since $P_{\Lambda'}-P_{\Lambda} = P_{\Lambda'-\Lambda}$ is a nonzero linear combination of the polynomials defining the $r$-regular factor $\cB$.
\end{proof}
The above claim allows us to bound the terms from \Cref{eq:mainfourier} corresponding to tuples $(\Lambda_1,\ldots, \Lambda_m)\in \Sigma^m$ with $\deg(P_{\Lambda_1})\leq d$. This is because for such terms $|\widehat{\Gamma}_1(\Lambda_1)|< \sigma$ by the above claim, and $|\widehat{\Gamma}_i(\Lambda_i)|\leq 1$ since $f_i$'s take values in $\D$. It remains to bound the terms for which $\deg(P_{\Lambda_1})> d$. We will need the following claim.
\begin{lemma}\label{lemma:lineardependence}
Assume that $L_1^{d+1}$ is not in the linear span of $L_2^{d+1},\ldots,L_m^{d+1}$, and let $(\Lambda_1,\ldots, \Lambda_m)\in \Sigma^m$ be such that $\deg(P_{\Lambda_1})\geq d+1$. Then
$$
\sum_{i=1}^m P_{\Lambda_i}(L_i(X))\not\equiv 0.
$$
\end{lemma}
This combined with \Cref{thm:nearorthogonality} implies that for a sufficiently fast growing choice of $r(\cdot)$, $\E_{X}\big[\expo{\sum_{i=1}^m P_{\Lambda_i}(L_i(X))}\big]<\sigma$ which completes the proof of \Cref{thm:gowerswolf}. Thus, we are left with proving~\Cref{lemma:lineardependence}.
\begin{proofof}{of \Cref{lemma:lineardependence}}
Assume to the contrary that $\sum_{i=1}^m P_{\Lambda_i}(L_i(X))\equiv 0$. Denoting the coordinates of $\Lambda_i$ by $(\lambda_{i,1},\ldots, \lambda_{i,C})\in \Sigma^C$ we have
$$
\sum_{i=1}^m P_{\Lambda_i}(L_i(X)) = \sum_{i\in [m], j\in [C]} \lambda_{i,j} P_j(L_i(X))\equiv 0.
$$
Since the polynomial factor defined by $P_1,\ldots, P_C$ is $r$-regular with a sufficiently fast growing growth function $r(\cdot)$, \Cref{thm:nearorthogonality} implies that for every $j\in [C]$ we must have that
\begin{equation}\label{eq:nonclassicalcancel}
\sum_{i=1}^m \lambda_{i,j}P_j(L_i(X)) \equiv 0.
\end{equation}
Since $\deg(P_{\Lambda_1})> d$, there must exist $j\in [C]$ such that $\lambda_{1,j}\neq 0$ and $\deg(\lambda_{1,j}P_j)>d$. Let $j^*\in [C]$ be such that $\deg(\lambda_{1,j^*}P_{j^*})$ is maximized, and let $d^*:=\deg(P_{j^*})$
(note that $\deg(\lambda_{1,j^*}P_{j^*}) \le d^*$). We will first prove that replacing $P_{j^*}$ with a classical homogeneous polynomial $Q$ of the same degree, \Cref{eq:nonclassicalcancel} for $j=j^*$ would still hold.

\begin{claim}\label{claim:replacePwithQ}
Let $Q:\F^n \to \T$ be a classical homogeneous polynomial with $\deg(Q)=d^*$. Then
$$
\sum_{i=1}^m \lambda_{i,j^*} Q(L_i(X)) \equiv 0.
$$
\end{claim}

\begin{proof}
Assume to the contrary that $\sum_{i=1}^m \lambda_{i,j^*} Q(L_i(X)) \not\equiv 0$. By \Cref{claim:lc1} for degree $d^*$ and linear forms $L_1,\ldots, L_m$, we can find a set of coefficients $\{a_{i,t} \in \Z\}_{i\in [\ell], t\in [m']}$, $\{c_{i,t}\in \F\backslash\{0\}\}_{i\in [\ell], t\in [m']}$ and linear forms $\{M_t\}_{t\in [m']}$ with $|M_t| \le d^*$ such that the first nonzero entry of every $M_t$ is equal to $1$, such that
\begin{equation}\label{eq:classicalcancel}
\sum_{i=1}^m \lambda_{i,j^*} Q(L_i(X))= \sum_{i\in [\ell],t\in [m']}a_{i,t}\lambda_{i,j^*} Q(c_{i,t}M_t(X))=
\sum_{t=1}^{m'} \alpha_t Q(M_t(X)) \not\equiv 0,
\end{equation}
where $\alpha_t=  \sum_{i=1}^\ell a_{i,t}\lambda_{i,j^*} |c_{i,t}|^{d^*}$. Here the second equality follows from $Q$ being classical and homogeneous. Furthermore,
\begin{equation}
\sum_{i=1}^m \lambda_{i,j^*} P_{j^*}(L_i(X))= \sum_{i\in [\ell], t\in [m']}a_{i,t}\lambda_{i,j^*} P_{j^*}(c_{i,t}M_t(X))= \sum_{t=1}^{m'}  \beta_t P_{j^*}(M_t(X)),
\end{equation}
where $\beta_t=\sum_{i=1}^{\ell}a_{i,t} \lambda_{i,j^*} \sigma_{i,t}$, $\{\sigma_{i,t}\}_{i\in [\ell], t\in [m']}$ are integers whose existence follows from the homogeneity of $P_{j^*}$. Moreover, by \cref{remark:homogeneity} we know that $\sigma_{i,t} \equiv |c_{i,t}|^{d^*}\mod p$, and hence
\begin{equation}
\alpha_t \equiv \beta_t \mod p, \quad \forall t \in [m'].
\end{equation}
Now notice that \Cref{eq:classicalcancel} implies that there exists some $t\in [m']$ for which $\alpha_{t}Q \neq 0$, which, since $Q$ is classical, is equivalent to $\alpha_{t} \not\equiv 0 \mod p$. Hence also $\beta_{t} \not\equiv 0 \mod p$. Let $T=\{t \in [m']: \beta_{t} \not \equiv 0 \mod p\}$, which we just verified is nonempty. Then for $t \in T$, $\deg(\beta_t P_{j^*})=\deg(P_{j^*})=d^*$; and for $t \in [m] \setminus T$, $\deg(\beta_t P_{j^*}) \le d^*-(p-1) < d^*$.

We can now decompose
\begin{equation}\label{eq:suminToutT}
\sum_{i=1}^m \lambda_{i,j^*} P_{j^*}(L_i(X)) = \sum_{t \in T}  \beta_t P_{j^*}(M_t(X)) + \sum_{t \in [m] \setminus T}  \beta_t P_{j^*}(M_t(X)).
\end{equation}
By \Cref{pro:nearorthog-niceform}, the first sum in \Cref{eq:suminToutT} is a nonzero polynomial of degree $d^*$, and by our previous argument, the sum for $t \not\in T$ is a polynomial of degree less than $d^*$. Hence, we get that
$\sum_{i=1}^m \lambda_{i,j^*} P_{j^*}(L_i(X))$ is a nonzero polynomial of degree $d^*$, which is a contradiction to our assumption.
\end{proof}

We have proved that for every choice of a degree-$d^*$ classical homogeneous polynomial $Q$, $\sum_{i=1}^m \lambda_{i,j^*} Q(L_i(X))\equiv 0$. Now choosing the polynomial $Q(x(1),\ldots,x(n))=x(1)\cdots x(d^*)$ and looking at the coefficients of the monomials of degree $d^*$, we have
$$
\sum_{i=1}^m \lambda_{i,j^*} L_i^{d^*} \equiv 0.
$$
Recalling that $\lambda_{1,j^*}\neq 0$ and that $d^*>d$, this means that $L_1^{d+1}$ can be written as a linear combination of $L_2^{d+1},\ldots, L_m^{d+1}$, a contradiction.
\end{proofof}

\bibliographystyle{amsalpha}
\bibliography{equidistbib}

\end{document}